\documentclass[11pt]{article}
    \usepackage{url}
    \usepackage{verbatim}
    \usepackage[titletoc]{appendix}
    \usepackage{graphicx}
	\usepackage{amsmath}
	\usepackage{amsthm}
	\usepackage{amsfonts}
	\usepackage{graphicx}
    \usepackage{nicefrac}
    \usepackage{longtable}
    \usepackage{color}
    \usepackage{graphicx, amssymb, subcaption,graphics}
    \usepackage{cite}
    \usepackage{epstopdf}

    \newcommand{\be}{\begin{equation}}
    \newcommand{\ee}{\end{equation}}

    \newcommand{\nrm}[1]{\left\| #1 \right\|}

    \def\f{\mbox{\boldmath $f$}}
    \def\x{\mbox{\boldmath $x$}}
    \def\v{\mbox{\boldmath $v$}}
    \def\w{\mbox{\boldmath $w$}}
    \newcommand\dt {{\Delta t}}
     \def\0{\mbox{\boldmath $0$}}

	\newtheorem{thm}{Theorem}[section]
	\newtheorem{prop}[thm]{Proposition}
	\newtheorem{cor}[thm]{Corollary}
	
	\newtheorem{lem}[thm]{Lemma}
	\newtheorem{rem}[thm]{Remark}
	
	\newtheorem{defi}{Definition}
		
\begin{document}
	\title{A third order exponential time differencing numerical scheme for no-slope-selection epitaxial thin film model with energy stability}
	
	\author{
Kelong Cheng \thanks{School of Science, Southwest University of Science and Technology,  Mianyang, Sichuan 621010, P. R. China (zhengkelong@swust.edu.cn)}
\and
Zhonghua Qiao\thanks{Department of Applied Mathematics, The Hong Kong Polytechnic University, Hung Hom, Kowloon, Hong Kong (zqiao@polyu.edu.hk)}
\and		
Cheng Wang\thanks{Department of Mathematics, The University of Massachusetts, North Dartmouth, MA  02747, USA (Corresponding Author: cwang1@umassd.edu)}
}

	\maketitle
	\numberwithin{equation}{section}

	\begin{abstract}
In this paper we propose and analyze a (temporally) third order accurate exponential time differencing (ETD) numerical scheme for the no-slope-selection (NSS) equation of the epitaxial thin film growth model, with Fourier pseudo-spectral discretization in space. A linear splitting is applied to the physical model, and an ETD-based multistep approximation is used for time integration of the corresponding equation. In addition, a third order accurate Douglas-Dupont regularization term, in the form of $-A \dt^2 \phi_0 (L_N)  \Delta_N^2 ( u^{n+1} - u^n)$, is added in the numerical scheme. A careful Fourier eigenvalue analysis results in the energy stability in a modified version, and a theoretical justification of the coefficient $A$ becomes available. As a result of this energy stability analysis, a uniform in time bound of the numerical energy is obtained. And also, the optimal rate convergence analysis and error estimate are derived in details, in the $\ell^\infty (0,T; H_h^1) \cap \ell^2 (0,T; H_h^3)$ norm, with the help of a careful eigenvalue bound estimate, combined with the nonlinear analysis for the NSS model. This convergence estimate is the first such result for a third order accurate scheme for a gradient flow. Some numerical simulation results are presented to demonstrate the efficiency of the numerical scheme and the third order convergence. The long time simulation results for $\varepsilon=0.02$ (up to $T=3 \times 10^5$) have indicated a logarithm law for the energy decay, as well as the power laws for growth of the surface roughness and the mound width. In particular, the power index for the surface roughness and the mound width growth, created by the third order numerical scheme, is more accurate than those produced by certain second order energy stable schemes in the existing literature.
	\end{abstract}
	
\noindent
{\bf Key words.} \, epitaxial thin film growth, slope selection, exponential time differencing,  energy stability, optimal rate convergence analysis, aliasing error
	
\medskip
	
\noindent
{\bf AMS Subject Classification} \, 35K30, 35K55, 65L06, 65M12, 65M70, 65T40

\section{Introduction}
In this article we consider an epitaxial thin film growth equation, which corresponds to the gradient flow associated with the following energy functional
\begin{equation}\label{energy-NSS}
E(u) := \int_{\Omega}\left(- \frac{1}{2} \ln(1+|\nabla u|^2) +\frac{\varepsilon^2}{2} |\Delta u|^2\right)\mbox{d} \textbf{x},
\end{equation}
where $\Omega = (0, L_x) \times (0, L_y)$, $u:\Omega\rightarrow \mathbb{R}$ is a periodic height function, and  $\varepsilon$ is a constant. In more details, the first non-quadratic term represents the Ehrlich-Schwoebel (ES) effect, according to which migrating adatoms must overcome a higher energy barrier to stick to a step from an upper rather than from a lower terrace~\cite{Ehrlich1966, libo06, libo03, Schwoebel1969}.  This results in an uphill atom current in the dynamics and the steepening of mounds in the film.  The second term, which is quadratic, but of higher-order, represents the isotropic surface diffusion effect~\cite{libo03, moldovan00}. In turn, the chemical potential becomes the following variational derivative of the energy
\begin{equation}  \label{NSS-chem pot}
 \mu := \delta_{u}E = \nabla\cdot\left(\frac{\nabla u}{1+|\nabla u|^2}\right) + \varepsilon^2 \Delta^2 u,
\end{equation}
and the no-slope-selection (NSS) equation stands for the $L^2$ gradient flow
\begin{equation}\label{equation-NSS}
\partial_tu = - \mu = -\nabla\cdot\left(\frac{\nabla u}{1+|\nabla u|^2}\right) - \varepsilon^2 \Delta^2 u .
\end{equation}

In the small-slope regime, where $|\nabla u|^2\ll 1$, (\ref{equation-NSS}) may be approximated as
	\begin{equation}
\partial_t u =  \nabla \cdot \left(|\nabla u|^2\nabla u \right) - \Delta u - \varepsilon^2 \Delta^2 u \ ,
	\label{equation-SS}
	\end{equation}
with the energy functional given by
\begin{equation}  \label{energy-SS}
E(u) = \int_{\Omega}\left( \frac14 ( | \nabla u |^2 -1)^2 +\frac{\varepsilon^2}{2} |\Delta u|^2\right)\mbox{d} \textbf{x} .
\end{equation}
This model is referred to as the slope-selection (SS) equation~\cite{kohn06, kohn03, libo03, moldovan00}.  A solution to (\ref{equation-SS}) exhibits pyramidal structures, where the faces of the pyramids have slopes $|\nabla u| \approx 1$; meanwhile, the no-slope-selection equation (\ref{equation-NSS}) exhibits mound-like structures, and the slopes of which (on an infinite domain) may grow unbounded~\cite{libo03, wang10}. On the other hand, both solutions have up-down symmetry in the sense that there is no way to distinguish a hill from a valley.  This can be altered by adding adsorption/desorption or other dynamics.


In \cite{libo03}, the global in time well-posedness for two nonlinear models of epitaxial thin film epitaxy, with or without slope selection, was established. And also, the gradient bound and the energy asymptotic law for the SS and NSS equations have been studied in~\cite{kohn03, libo04, LiD2017c}, as $\varepsilon \to 0$. In addition, the large-system asymptotic form of the minimum energy and the magnitude of gradients of energy-minimizing surfaces for epitaxial growth models are analyzed in~\cite{libo06}, with infinite or finite ES barrier. Specially, for the case of a finite ES effect (corresponding to the model in this article), the well-posedness of the initial-boundary-value problem is proved and the bounds for the scaling laws of interface width, surface slope and energy are obtained.

There have been many efforts to devise and analyze numerical schemes for both the SS and NSS equations; see the related references~\cite{chenwang12, libo03, qiao12, qiao12b, xu06}, etc. In particular, the numerical schemes with high order accuracy and energy stability have been of great interests, due to the long time nature of the gradient flow coarsening process. 
Among the energy stable numerical approaches, the idea of convex splitting has attracted many attentions. For the epitaxial thin film growth models, the first such work was reported in~\cite{wang10}, in which the authors studied unconditionally energy stable schemes, based on the convex-concave decomposition of the energy, motivated by Eyre's pioneering work~\cite{eyre98}. Some other related developments on the energy stable schemes for the MBE models could be found in~\cite{chen12, chen14, feng17a, Ju17, LiD2016a, LiW18, LiX17, qiao15, qiao17, shen12, yang17b}, etc. In particular, it is worthy of mentioning the works~\cite{chen12, LiW18}, in which the authors proposed linear numerical schemes for the NSS equation, with first and second order temporal accuracy orders, respectively, so that the energy stability could be established at a theoretical level. In fact, the following subtle fact has played an essential role in the nonlinear energy stability analysis: in spite of its complicated form in the denominator, the nonlinear term in the NSS equation~\eqref{equation-NSS} has automatically bounded higher order derivatives in the $L^\infty$ norm. In addition, extensive convergence analysis works have been undertaken for these various energy stable numerical schemes, for both the SS and NSS equations.

  On the other hand, it is observed that, a theoretical analysis of a (temporally) third order accurate numerical scheme for the gradient equations remains an open problem. In this article, we propose and analyze a (temporally) third order accurate  numerical scheme for the NSS equation~\eqref{equation-NSS}, based on the exponential time differencing (ETD) temporal algotirhm, combined with Fourier pseudo-spectral approximation in space. In general, an exact integration of the linear part of the NSS equation is involved in the ETD-based scheme, followed by multi-step explicit approximation of the temporal integral of the nonlinear term~\cite{Suli14, Beylkin98, Cox02, Hochbruck10, Hochbruck11}. An application of such an idea to various gradient models has been reported in recent works~\cite{Ju17, Ju14, Ju15a, Ju15b, wangx16, zhu16}, with the high order accuracy and preservation of the exponential behavior observed in the numerical experiments. At the theoretical side, some related stability and convergence analyses have also been reported for a few first and second order accurate numerical algorithms, while a theoretical justification for the third order one has not been available. To overcome such a difficulty, we make use of an alternate splitting idea, as reported in~\cite{Ju17}: to combine the surface diffusion term with an auxiliary linear diffusion term, and make the corresponding revision in the explicit extrapolation part. As a result of this splitting, the energy stability for the first order ETD scheme has been established in the reported work.

Meanwhile, it is observed that, a theoretical justification of the energy stability for the higher-order ETD-based schemes becomes very challenging, due to the explicit treatment of the nonlinear terms, as well as their multi-step nature. In the alternate energy inequality derived for the second order ETD-based scheme, as reported in~\cite{Ju17}, some positive increase of the numerical energy becomes possible, so that a uniform-in-time bound for the numerical energy is not theoretically available any more. To overcome this difficulty, we add a third order Douglas-Dupont regularization term in the ETD-based scheme, namely in the form of $-A \dt^2 \phi_0 (L_N)  \Delta_N^2 ( u^{n+1} - u^n)$. Furthermore, a careful eigenvalue analysis in the Fourier space enables us to derive a rigorous stability estimate for a modified energy function, which contains the original energy functional and a few non-negative numerical correction terms. As a result of this modified energy stability, we are able to derive a uniform-in-time bound for the original energy functional bound. 

  Moreover, we provide a theoretical proof of an $O (\dt^3 + h^m)$ rate convergence estimate for the proposed third order ETD-based scheme, in the $\ell^\infty(0,T; H_h^1) \cap \ell^2 (0, T; H_h^3)$ norm. To obtain such a convergence analysis, we have to decompose the numerical scheme into two stages: the exponential integration for the linear part is considered in the first stage, with an intermediate variable introduced, and the explicit multi-step extrapolation for the nonlinear part is involved in the second stage. Error estimates are carried out in both stages, with extensive applications of linearized stability analysis in the second stage. One key difficulty in the analysis for higher order ETD-based numerical scheme is associated with various global operators involved in the algorithm, as well as their inverse operators. To derive a uniform bound for these operators, we perform careful eigenvalue estimates for these operators, as well as their composition. In addition, 
an aliasing error control technique has to be utilized in the $\ell^\infty (0, T, H_h^1)$ error estimate, combined with extensive scaling law arguments between $\dt$ and $h$. As a result of these careful estimates, the derived convergence estimate becomes unconditional, i.e., no scaling law between $\dt$ and $h$ is needed to ensure the convergence result. To our knowledge, it is the first such result for a third order accurate scheme for a gradient flow.

The long time simulation results for the coarsening process have indicated a logarithm law for the energy decay, as well as the power laws for growth of the surface roughness and the mound width. In particular, the power index for the surface roughness and the mound width growth, created by the proposed third order ETD-based scheme, is more accurate than those created by certain second order schemes in the existing literature, with the same numerical resolution. This experiment has demonstrated the robustness of the proposed third order numerical scheme.

The rest of the article is organized as follows. In Section~\ref{sec-num scheme} we present the numerical scheme. First we review the Fourier pseudo-spectral approximation in space and certain technical lemma to control the aliasing error. Afterward, the third order ETD-based scheme is introduced, and a modified energy stability is established.   Subsequently, the $\ell^\infty(0,T; H_h^1) \cap \ell^2 (0, T; H_h^3)$ convergence estimate is provided in Section~\ref{sec-convergence}. In Section~\ref{sec:numerical results} we present the numerical results, including the accuracy test and the long time simulation for the coarsening process.   Finally, the concluding remarks are given in Section~\ref{sec:conclusion}.


	\section{The numerical scheme}
	\label{sec-num scheme}
	
\subsection{Review of the Fourier pseudo-spectral approximation}

For simplicity of presentation, we assume that the domain is given by $\Omega = (0,L)^2$, $N_x = N_y = N$ and $N \cdot h = L$. A more general domain could be treated in a similar manner. Furthermore, to facilitate the pseudo-spectral analysis in later sections, we set $N = 2K+1$. All the variables are evaluated at the regular numerical grid $(x_i, y_j)$, with $x_i = i h$, $y_j=jh$, $0 \le i , j \le 2K +1$.

Without loss of generality, we assume that $L=1$. For a periodic function $f$ over the given 2-D numerical grid, set its discrete Fourier expansion as
\begin{equation}
  f_{i,j} = \sum_{k,\ell=-K}^{K}
   \hat{f}_{k,\ell} \exp \left( 2 \pi {\rm i} ( k x_i + \ell y_j ) \right) ,
   \label{spectral-coll-1}
\end{equation}
its collocation Fourier spectral approximations to first and second order partial derivatives in the $x$-direction become
\begin{eqnarray}
  \left( {\cal D}_{Nx} f \right)_{i,j} = \sum_{k,\ell=-K}^{K}
   \left( 2 k \pi {\rm i} \right) \hat{f}_{k,\ell}
   \exp \left( 2 \pi {\rm i} ( k x_i + \ell y_j ) \right) ,
   \label{spectral-coll-2-1}
\\
  \left( {\cal D}_{Nx}^2 f \right)_{i,j} = \sum_{k,\ell=-K}^{K}
   \left( - 4 \pi^2 k^2 \right) \hat{f}_{k,\ell}
   \exp \left( 2 \pi {\rm i} ( k x_i + \ell y_j) \right) .
   \label{spectral-coll-2-3}
\end{eqnarray}
The differentiation operators in the $y$ direction, namely, ${\cal D}_{Ny}$ and ${\cal D}_{Ny}^2$, could be defined in the same fashion. In turn, the discrete Laplacian, gradient
and divergence become
\begin{eqnarray}
  \Delta_N f =  \left( {\cal D}_{Nx}^2  + {\cal D}_{Ny}^2 \right) f ,  \nonumber
\\
  \nabla_N f = \left(  \begin{array}{c}
  {\cal D}_{Nx} f  \\
  {\cal D}_{Ny} f
  \end{array}  \right)  ,  \quad
  \nabla_N \cdot \left(  \begin{array}{c}
  f _1 \\
  f _2
  \end{array}  \right)  = {\cal D}_{Nx} f_1 + {\cal D}_{Ny} f_2 ,
  \label{spectral-coll-3}
\end{eqnarray}
at the point-wise level. 
See the derivations in the related references~\cite{Boyd2001, canuto82, Gottlieb1977}, etc.

In addition, with an introduced operator $L_N = \varepsilon^2 \Delta_N^2 - \kappa \Delta_N$, which will be repeatedly used in this work, the operators 
$(\dt L_N)^{-1}$ and ${\rm e}^{- \dt L_N}$ are defined as
\begin{eqnarray}
  &&
  \left( (\dt L_N)^{-1} f \right)_{i,j} = \sum_{k,\ell=-K}^{K}
   \frac{1}{\dt \Lambda_{k,\ell}} \hat{f}_{k,\ell}
   \exp \left( 2 \pi {\rm i} ( k x_i + \ell y_j) \right) ,  \label{spectral-coll-4-2}
 \\
  &&
  \left( {\rm e}^{- \dt L_N} f \right)_{i,j} = \sum_{k,\ell=-K}^{K}
   {\rm e}^{- \dt \Lambda_{k,\ell}} \hat{f}_{k,\ell}
   \exp \left( 2 \pi {\rm i} ( k x_i + \ell y_j) \right) ,  \label{spectral-coll-4-3}
\\
  &&  \mbox{with} \quad
  \Lambda_{k,\ell} = \varepsilon^2 \lambda_{k,\ell}^2 + \kappa \lambda_{k,\ell} ,  \quad
  \lambda_{k,\ell} = (2 k \pi)^2 + (2 \ell \pi)^2 ,
  \label{spectral-coll-4-4}
\end{eqnarray}
for a grid function $f$ with the discrete Fourier expansion as (\ref{spectral-coll-1}), with a zero-mean: $\overline{f} := h^2 \sum_{i,j=0}^{N-1}  f_{i,j} = 0$ (so that $\hat{f}_{0,0}=0$). Similarly, for a zero-mean grid function $f$, the operator $( I - {\rm e}^{-\dt L_N} )^{-1}$ is defined as
\begin{eqnarray}
  \left( ( I - {\rm e}^{-\dt L_N} )^{-1} f \right)_{i,j} = \sum_{k,\ell \ne \0}
   \frac{1}{1 - {\rm e}^{- \dt \Lambda_{k,\ell}} } \hat{f}_{k,\ell}
   \exp \left( 2 \pi {\rm i} ( k x_i + \ell y_j) \right) .
   \label{spectral-coll-5}
\end{eqnarray}

  Given any periodic grid functions $f$ and $g$ (over the 2-D numerical grid), the spectral approximations to the $L^2$ inner product and $L^2$ norm are introduced as
\begin{eqnarray}
  \left\| f \right\|_2 = \sqrt{ \left\langle f , f \right\rangle } ,  \quad \mbox{with} \quad
  \left\langle f , g \right\rangle  = h^2 \sum_{i,j=0}^{N -1}   f_{i,j} g_{i,j} .
  \label{spectral-coll-inner product-1}
\end{eqnarray}
A careful calculation yields the following formulas of summation by parts at the discrete level (see the related discussions~\cite{chen12, chen14, gottlieb12a, gottlieb12b}):
\begin{eqnarray}
  \left\langle f ,  \Delta_N  g  \right\rangle
  = - \left\langle \nabla_N f ,  \nabla_N g   \right\rangle  ,    \quad
  \left\langle f ,  \Delta_N^2  g  \right\rangle
  =  \left\langle \Delta_N f ,  \Delta_N g   \right\rangle  .
  \label{spectral-coll-inner product-3}
\end{eqnarray}



In addition to the standard $\ell^2$ norm, we also introduce the $\ell^p$ and discrete maximum norms for a grid function $f$, to facilitate the analysis in later sections:
\begin{equation}
 \nrm{f}_{\infty} := \max_{i,j} |f_{i,j}| ,   \qquad
 \nrm{f}_{p}  := \Bigl( \sum_{i,j=0}^{N-1} |f_{i,j} |^p \Bigr)^{\frac{1}{p}} , \quad 1\leq p < \infty.  \label{spectral-defi-Lp}
\end{equation}
Moreover, for any numerical solution $\phi$, the discrete energy is defined as
	\be
E_N (\phi) = E_{c,1,N} (\phi) + \frac{\varepsilon^2}{2} \nrm{\Delta_N \phi}_2^2 \  , \quad E_{c,1,N} (\phi) = h^2 \sum_{i,j=0}^{N-1} \left( - \frac12 \ln \left( 1 + \left| \nabla_N \phi \right|^2 \right)_{i,j} \right) .
	\label{energy-discrete-spectral}
	\ee
	
Meanwhile, an appearance of aliasing error in the nonlinear term poses a serious challenge in the numerical analysis of Fourier pseudo-spectral scheme. To overcome such a well-known difficulty, we introduce a periodic extension of a grid function and a Fourier collocation interpolation operator.

\begin{defi}
  For any periodic grid function $f$ defined over a uniform 2-D numerical grid,
we denote $f_N$ as its periodic extension. In more detail,
assume that the grid function $f$ has a discrete Fourier expansion as
(\ref{spectral-coll-1}),
its continuous extension (projection) into ${\cal P}_K$ (the space of trigonometric polynomials of degree at most $K$) is given by
\begin{equation}
   f_N (\x)  = \sum_{k,\ell=-K}^{K}
   \hat{f}_{k,\ell}
     \exp \left( 2 \pi {\rm i} ( k x + \ell y) \right) .
    \label{spectral-coll-projection-2}
\end{equation}
And also, for any periodic continuous function $\f$, which may contain larger wave length, its collocation interpolation operator is defined as
\begin{eqnarray}
  & &
  \f_{i,j}  = \sum_{k,\ell=-K}^{K}  (\hat{f}_c)_{k,\ell}
     \exp \left( 2 \pi {\rm i} ( k x_i + \ell y_j) \right) ,   \nonumber
\\
  & &
    P_c^N \f_N (\x)  = \sum_{k,\ell=-K}^{K}  (\hat{f}_c)_{k,\ell}
     \exp \left( 2 \pi {\rm i} ( k x + \ell y) \right) ,
    \label{spectral-coll-projection-3}
\end{eqnarray}
in which the Fourier collocation coefficients $(\hat{f}_c)_{k, \ell}$ could be obtained by discrete Fourier transformation. Notice that $\hat{f}_c$ may not be the Fourier coefficients of $\f$, due to the truncation and aliasing errors.
\end{defi}

To overcome a key difficulty associated with the $H^m$ bound of the nonlinear term obtained by collocation interpolation, the following lemma is introduced. In fact, the case of $k_0=0$ was proven in earlier works~\cite{E92, E93}. The case of $k_0 \ge 1$ was analyzed in a recent article~\cite{gottlieb12b}.

\begin{lem} \label{lemma:aliasing error-1}
For any $\varphi \in {\cal P}_{m K}$ in dimension $d$, we have
\begin{equation}
  \left\| P_c^N \varphi \right\|_{H^{k_0}}
  \le  \left( \sqrt{m} \right)^d  \left\|  \varphi \right\|_{H^{k_0}} ,
  \quad \forall k_0 \in \mathbb{Z} ,  \, k_0 \ge 0 .
   \label{spectral-coll-projection-4}
\end{equation}
\end{lem}

On the other hand, for $f \notin {\cal P}_{m K}$, which may come from the nonlinearity in the denominator (such as the NSS model), the following aliasing error control estimate has to be applied. This inequality has been derived in an earlier work~\cite{canuto82}; we cite the result here.

\begin{lem} \label{lemma:aliasing error-2}
As long as $f$ and all its derivatives (up to $m$-th order) are continuous and periodic on $\Omega$, the convergence of the derivatives of the interpolation is given by
\begin{eqnarray}
  \| f - P_c^N f \|_{H^k}
  \leq C \| f \|_{H^m} h^{m-k} , \quad
\mbox{for} \, \, \, 0 \le k \le m ,  \, m > \frac{d}{2} .
   \label{spectral-coll-projection-5}
\end{eqnarray}
\end{lem}

\subsection{The proposed third order ETD-based numerical scheme}

In the derivation of ETD-based numerical schemes, we rewrite the NSS equation in the operator form as
\begin{equation}\label{equation-NSS-operator}
\partial_t u = - L u - f (u) ,  \quad \mbox{with} \, \, \,
L u = \varepsilon^2 \Delta^2 u - \kappa \Delta u , \, \,
f (u) = \nabla\cdot\left(\frac{\nabla u}{1+|\nabla u|^2}\right) + \kappa \Delta u .
\end{equation}
Moreover, with the Fourier pseudo-spectral spatial approximation, the corresponding system could be expressed as
\begin{equation}\label{equation-NSS-operator-2}
d_t u = - L_N u - f_N (u) ,  \quad \mbox{with} \, \, \,
L_N u = \varepsilon^2 \Delta_N^2 u - \kappa \Delta_N u , \, \,
f_N (u) = \nabla_N \cdot \left( \frac{\nabla_N u}{1+|\nabla_N u|^2}\right) + \kappa \Delta_N u .
\end{equation}
In fact, by using the integrating factor, an update of the exact solution from time instant $t^n$ to the next time step could be represented as
\begin{equation}
\label{ETD-0}
u (t_{n+1})={\rm e}^{-\dt L_N}u(t_n) - \int_0^{\Delta t} {\rm e}^{-(\dt - \tau) L_N}f_N (u(t_n+\tau) )\, d \tau.
\end{equation}

  We denote $u^n$ as the numerical approximation to the PDE solution at time step $t^n := n \dt$, with any integer $n$. With an application of multi-step Lagrange extrapolation formulas, required by the given accuracy order, we propose a third order ETD-based scheme for the NSS equation~\eqref{equation-NSS}:
\begin{eqnarray}
   u^{n+1} &=& {\rm e}^{- L_N \dt} u^n - A \dt^3 \phi_0 (L_N) \Delta_N^2 (u^{n+1} - u^n) - \dt \phi_0 (L_N) f_N (u^n)  \nonumber
 \\
   &&
   - \dt \phi_1 (L_N) ( \frac32 f_N (u^n) - 2 f_N (u^{n-1}) + \frac12 f_N (u^{n-2}) )  \nonumber
\\
  &&
  - \dt \phi_2 (L_N) ( \frac12 f_N (u^n) - f_N (u^{n-1}) + \frac12 f_N (u^{n-2}) ) ,
  \label{scheme-ETD-3rd-0}
\end{eqnarray}
with
\begin{eqnarray}
  &&
  \phi_0 (L_N) = ( \dt L_N)^{-1} ( I - {\rm e}^{-\dt L_N} ) ,   \nonumber
\\
  &&
  \phi_1 (L_N) = ( \dt L_N)^{-1} ( I - (\dt L_N)^{-1} (I - {\rm e}^{-\dt L_N}) )  , \nonumber
\\
  &&
   \phi_2 (L_N) = ( \dt L_N)^{-1} \left( I - 2 (\dt L_N)^{-1} (I - (\dt L_N)^{-1} (I - {\rm e}^{- \dt L_N}) ) \right)  .
  \label{scheme-ETD-3rd-1}
\end{eqnarray}

\begin{rem}
Without an artificial regularization term, the first, second and third order multi-step ETD-based schemes have been analyzed in~\cite{Ju17}:
\begin{eqnarray}
  &&
  {\bf the \, \, first\, \, order \, \, scheme \, \, (ETD1):}  \, \,  \,
  u^{n+1}={\rm e}^{-\dt L_N} u^n - \dt \phi_0 (L_N) f_N (u^n) ,  \label{scheme-ETD-1st}
\\
  &&
  {\bf the \, \, second \, \, order \, \, scheme \, \, (ETDMs2):}  \nonumber
\\
  &&  \quad
  u^{n+1}={\rm e}^{-\dt L_N} u^n - \dt \phi_0 (L_N) f_N (u^n) - \dt \phi_1 (L_N) ( f_N(u^n) - f_N (u^{n-1})) , \label{scheme-ETD-2nd}
\\
  &&
  {\bf the \, \, third \, \, order \, \, scheme \, \, (ETDMs3):}  \nonumber
\\
  &&  \quad
  u^{n+1}={\rm e}^{-\dt L_N} u^n - \dt \phi_0 (L_N) f_N (u^n)
  - \dt \phi_1 (L_N) ( \frac32 f_N (u^n) - 2 f_N (u^{n-1}) + \frac12 f_N (u^{n-2}) )  \nonumber
\\
  && \qquad  \qquad  \, \,
  - \dt \phi_2 (L_N) ( \frac12 f_N (u^n) - f_N (u^{n-1}) + \frac12 f_N (u^{n-2}) ) .  \label{scheme-ETD-3rd-Ju17}
\end{eqnarray}
In comparison with~\eqref{scheme-ETD-3rd-Ju17} reported in~\cite{Ju17}, we have added a third order Douglas-Dupont regularization term, namely, $- A \dt^2 \Delta_N^2 (u^{n+1} - u^n)$, in the proposed third order numerical scheme. Such a regularization term enables one to theoretically justify the energy stability, as will be demonstrated in later sections.
\end{rem}

For the proposed scheme~\eqref{scheme-ETD-3rd-0}, the mass-conservative property is always valid: $\overline{u^{n+1}} = \overline{u^n} = \overline{u^0} := \beta_0$, which comes from the following identities:
\begin{eqnarray}
  &&
  \overline{{\rm e}^{-\dt L_N} u^n} = \overline{u^n} ,  \quad \mbox{since} \, \, \,
   {\rm e}^{-\dt \Lambda_{0,0}} \equiv 1 ,
\\
  &&
  \overline{f_N (u^k)} = \overline{ \nabla_N \cdot \left( \frac{\nabla_N u^k}{1+|\nabla_N u^k|^2}\right) } = 0 ,  \quad \forall k ,
\\
  &&
  \overline{\phi_j (L_N) (g)} = 0 ,  \quad \mbox{for any $g$ with $\overline{g} = 0$} , \quad j=0, 1, 2 .
\end{eqnarray}

\subsection{Some preliminary estimates}

To facilitate the stability and convergence analysis for the numerical error function, we introduce the following linear operators:
\begin{eqnarray}
  {\cal G}_N &=& \left( \phi_0 (L_N) \right)^{-1} = \dt L_N ( I - {\rm e}^{-\dt L_N} )^{-1} , \label{operator-1}
\\
  G^{(1)}_N &=& \left( \phi_0 (L_N) \right)^{-1}  \phi_1 (L_N)
  =  ( I - {\rm e}^{-\dt L_N} )^{-1}  ( I - (\dt L_N)^{-1} (I - {\rm e}^{-\dt L_N}) )  ,
  \label{operator-2}
\\
  G^{(2)}_N &=& \left( \phi_0 (L_N) \right)^{-1}  \phi_2 (L_N)   \nonumber
\\
  &=&
   ( I - {\rm e}^{-\dt L_N} )^{-1}  \left( I - 2 (\dt L_N)^{-1} (I - (\dt L_N)^{-1} (I - {\rm e}^{-\dt L_N}) ) \right) . \label{operator-3}
\end{eqnarray}
In more details, for any grid function $f$ with the following discrete Fourier expansion:
\begin{equation}
  f_{i,j} = \sum_{k,\ell=-K}^K \hat{f}_{k,\ell} {\rm e}^{2 \pi i ( k x_i + \ell y_j)/L} ,
  \label{Fourier-1}
\end{equation}
an application of the above operators become
\begin{eqnarray}
  &&
  ( {\cal G}_N f )_{i,j} = \sum_{k,\ell=-K}^K  \frac{\dt \Lambda_{k,\ell}}{1
  - {\rm e}^{- \dt \Lambda_{k,\ell}} }
  \hat{f}_{k,\ell} {\rm e}^{2 \pi i ( k x_i + \ell y_j)/L} ,   \label{Fourier-2-1}
\\
  &&
  ( G^{(1)}_N f )_{i,j} = \sum_{k,\ell=-K}^K  \frac{1 - \frac{1 - {\rm e}^{- \dt \Lambda_{k,\ell}} }{\dt \Lambda_{k,\ell}} }{1 - {\rm e}^{- \dt \Lambda_{k,\ell}} }
  \hat{f}_{k,\ell} {\rm e}^{2 \pi i ( k x_i + \ell y_j)/L} ,   \label{Fourier-2-2}
\\
  &&
  ( G^{(2)}_N f )_{i,j} = \sum_{k,\ell=-K}^K  \frac{1 - 2 \frac{1 - \frac{1 - {\rm e}^{- \dt \Lambda_{k,\ell}} }{\dt \Lambda_{k,\ell}} }{ \dt \Lambda_{k,\ell}} }{1 - {\rm e}^{- \dt \Lambda_{k,\ell}} }
  \hat{f}_{k,\ell} {\rm e}^{2 \pi i ( k x_i + \ell y_j)/L} ,   \label{Fourier-2-3}
\end{eqnarray}
with $\lambda_{k,\ell}$, $\Lambda_{k,\ell}$ given by~\eqref{spectral-coll-4-4}. Meanwhile, since all the eigenvalues in (\ref{Fourier-2-1}), $\frac{\dt \Lambda_{k,\ell}}{1
  - {\rm e}^{- \dt \Lambda_{k,l}} } $, are non-negative, we define ${\cal G}^{(0)}_N = ( {\cal G}_N)^{1/2}$ as
\begin{eqnarray}
  ( {\cal G}^{(0)}_N f )_{i,j} = ( {\cal G}_N )^{1/2} f _{i,j} = \sum_{k,\ell=-K}^K  \left( \frac{\dt \Lambda_{k,\ell}}{1 - {\rm e}^{- \dt \Lambda_{k,\ell}} }   \right)^{\frac12}
  \hat{f}_{k,\ell} {\rm e}^{2 \pi ( k x_i + \ell y_j)/L}  .
    \label{Fourier-3}
\end{eqnarray}
Obviously, the operator ${\cal G}^{(0)}_N$ is commutative with any differential operator in the Fourier collocation spectral space, and the following summation by parts formula is available:
\begin{eqnarray}
   \left\langle f , {\cal G}_N g \right\rangle = \left\langle {\cal G}^{(0)}_N f , {\cal G}^{(0)}_N g \right\rangle .   \label{Fourier-4}
\end{eqnarray}

  In fact, if we denote
\begin{eqnarray}
  g_0 (x) = \frac{1 - {\rm e}^{-x}}{x} ,   \quad
  g_1 (x) = \frac{1 - \frac{1 - {\rm e}^{-x}}{x} }{x} ,  \quad
  g_2 (x) = \frac{1 - 2 \frac{1 - \frac{1 - {\rm e}^{-x}}{x} }{x} }{x}  ,   \label{lem 1-0}
\end{eqnarray}
for $x > 0$, the following result will be used in later analysis.

\begin{lem}  \label{lem:lem 1}
 (1) $g_i (x)$ is decreasing, for $i=0, 1, 2$. 

 (2) $\frac{g_1 (x)}{g_0 (x)} \le \frac{1}{1 - {\rm e}^{-2} }$ and $\frac{g_2 (x)}{g_0 (x)} \le \frac{1}{1 - {\rm e}^{-2} }$, $\forall x > 0$.
\end{lem}

As a result, the following estimates could be derived, using a careful Fourier analysis.

\begin{prop}  \label{prop:prop 1}
  For any periodic grid function $f$ with $\overline{f} =0$, we have
\begin{eqnarray}
  &&
  \| f \|_2 \le \nrm{ {\cal G}^{(0)}_N f }_2 \le C_1 ( \| f \|_2 + \dt^{1/2} ( \varepsilon \| \Delta_N f \|_2 + \kappa^{1/2} \| \nabla_N f \|_2 ) ) ,  \label{prop-1-0-1}
\\
  &&
  \dt \langle L_N f , f \rangle  \le \langle {\cal G}_N f , f \rangle ,  \quad
  \left\langle L_N  f , - \Delta_N {\rm e}^{- L_N \dt} f  \right\rangle \ge 0  ,
  \label{prop-1-0-1-2}
\\
  &&
   \| f \|_{-1} \le \nrm{ {\cal G}^{(0)}_N f }_{-1} \le C_2 ( \| f \|_{-1} + \dt^{1/2} ( \varepsilon \| \nabla_N f \|_2 + \kappa^{1/2} \| f \|_2 ) )  ,  \label{prop-1-0-2}
\\
  &&
   \| \nabla_N f \|_2 \le \nrm{ \nabla_N {\cal G}^{(0)}_N f }_2 \le C_3 ( \| \nabla_N f \|_2 + \dt^{1/2} \varepsilon \| \nabla_N \Delta_N f \|_2 )  ,  \label{prop-1-0-3}
\\
  &&
  \| G^{(1)}_N f \|_2  \le C_4 \| f \|_2 ,   \quad  \| G^{(2)}_N f \|_2  \le C_5 \| f \|_2 , \quad
  C_4 = C_5 =  \frac{1}{1 - {\rm e}^{-2}} ,
   \label{prop-1-0-4}
\end{eqnarray}
in which the constants $C_i$, $1 \le i \le 5$, are only dependent on $\Omega$, independent on $f$ and $N$.
\end{prop}

The detailed proof for both Lemma~\ref{lem:lem 1} and Proposition~\ref{prop:prop 1} will be provided in Appendices~\ref{proof:Lemma 1} and \ref{proof:Prop 1}, respectively. In addition, we denote
\begin{equation}
  g_N (u) =  \frac{\nabla_N u}{1+|\nabla_N u|^2} + \kappa \nabla_N u ,  \label{defi-g}
\end{equation}
so that $f_N (u) = \nabla_N \cdot g_N (u)$.

\subsection{The energy stability analysis}

We choose $\kappa_0 = \frac18$ and $\kappa \ge \frac14$ so that
\begin{equation}
  \kappa^* = \kappa - \kappa_0 \ge \frac12 \kappa . \label{relation-kappa}
\end{equation}
And also, we denote the following constants:
\begin{eqnarray}
  &&
   \gamma_1^{(0)} = \frac32 C_4 (1 + \kappa) , \, \, \,  \gamma_2^{(0)} = C_4 (1 + \kappa)  , \, \, \,  \gamma_3^{(0)} = \frac12 C_4 (1 + \kappa) ,  \label{constants-1}
\\
  &&
   \gamma^{(0)} = \gamma_1^{(0)} +  \gamma_2^{(0)} + \gamma_3^{(0)} =  3 C_4 (1 + \kappa)  ,  \quad \mbox{with $C_4 =  \frac{1}{1 - {\rm e}^{-2}}$} .   \label{constants-2}
\end{eqnarray}
In addition, the constant $\alpha_0 > 0$ is introduced to be the unique solution of the following equation:
\begin{eqnarray}
  \frac{ {\rm e}^{- \alpha_0} }{ 1 - {\rm e}^{- \alpha_0} }  + \frac12 = \frac{\gamma^{(0)} }{\kappa} ,  \quad \mbox{i.e.} \, \, \, \alpha_0 = \ln \Bigl( \frac{\gamma^{(0)} + \frac12 \kappa}{\gamma^{(0)} - \frac12 \kappa} \Bigr) .
  \label{constants-3}
\end{eqnarray}

The following two preliminary estimates in~\cite{Ju17} will be useful in the energy stability analysis.

\begin{lem} \cite{Ju17} \label{lem:vector inequality}
Denote a mapping ${\bf \beta}: R^2 \to R^2$: ${\bf \beta} (\v) = \frac{\v}{1 + | \v|^2 }$. Then we have
\begin{eqnarray}
  | {\bf \beta} (\v) - {\bf \beta} (\w) | \le | \v - \w | ,  \quad \forall \v, \, \w \, \in R^2 .
  \label{vector inequality-0}
\end{eqnarray}
\end{lem}

\begin{lem} \cite{Ju17} \label{lem:convexity}
Define $H (a,b) = \frac12 \ln (1 + a^2 + b^2) + \frac{\kappa_0}{2} (a^2 + b^2)$. Then $H (a,b)$ is convex in $R^2$ if and only if $\kappa_0 \ge \frac18$.
\end{lem}

As a result of these convexity result, we are able to obtain the following energy estimate.

\begin{lem} \label{lem: energy stab-0}
For the numerical solutions $u^{n+1}$ and $u^n$ with $\overline{u^{n+1}} = \overline{u^n}$, we have
\begin{eqnarray}
   E_N (u^{n+1}) - E_N (u^n) &\le& \langle L_N u^{n+1} + f_N (u^n), u^{n+1} - u^n \rangle  \nonumber
 \\
   &&
   - \frac{\varepsilon^2}{2} \| \Delta_N (u^{n+1} - u^n) \|_2^2
   - \kappa^* \| \nabla_N (u^{n+1} - u^n ) \|_2^2 . \label{energy stab-0}
\end{eqnarray}
\end{lem}

\begin{proof}
For simplicity of presentation, we denote
\begin{eqnarray}
  f_N^{(0)} (u) = \nabla_N \cdot \left( \frac{\nabla_N u}{1+|\nabla_N u|^2}\right) + \kappa _0 \Delta_N u ,  \quad \mbox{so that} \, \, \,
  f_N (u) = f_N^{(0)} (u) + \kappa^* \Delta_N u .  \label{energy stab-1}
\end{eqnarray}
By Lemma~\ref{lem:convexity}, $f_N^{(0)} (u)$ corresponds to a concave energy functional, so that the following convexity inequality is valid:
\begin{eqnarray}
  \langle f_N^{(0)} (u^n) ,  u^{n+1} - u^n \rangle \ge H_{N} (u^{n+1}) - H_N (u^n) , \quad \mbox{with} \, \, \, H_N (\phi) =  E_{c,1,N} (\phi) - \frac{\kappa_0}{2} \| \nabla_N \phi \|_2^2 , \label{energy stab-2}
\end{eqnarray}
with the nonlinear energy functional $E_{c,1,N} (\phi)$ defined in~\eqref{energy-discrete-spectral}. As a consequence, we get
\begin{eqnarray}
  &&
 \langle f_N (u^n) ,  u^{n+1} - u^n \rangle
 = \langle f_N^{(0)} (u^n) ,  u^{n+1} - u^n \rangle
 + \kappa^* \langle \Delta_N u^n ,  u^{n+1} - u^n \rangle \nonumber
\\
  &\ge&
  E_{c,1,N} (u^{n+1}) - E_{c,1,N} (u^n)  - \frac{\kappa_0}{2} ( \| \nabla_N u^{n+1} \|_2^2 - \| \nabla_N u^n \|_2^2 )   \nonumber
\\
  &&
    - \frac{\kappa^*}{2} ( \| \nabla_N u^{n+1} \|_2^2 - \| \nabla_N u^n \|_2^2 ) + \frac{\kappa^*}{2} \| \nabla_N (u^{n+1} - u^n ) \|_2^2 .  \label{energy stab-3}
\end{eqnarray}
Meanwhile, for the linear diffusion term $L_N$, the following equalities are available
\begin{eqnarray}
  &&
  \varepsilon^2 \langle \Delta_N^2 u^{n+1} ,  u^{n+1} - u^n \rangle
  =  \frac{\varepsilon^2}{2} ( \| \Delta_N u^{n+1} \|_2^2 - \| \Delta_N u^n \|_2^2 ) + \frac{\varepsilon^2}{2} \| \Delta_N (u^{n+1} - u^n ) \|_2^2 ,  \label{energy stab-4-1}
\\
  &&
  - \kappa \langle \Delta_N u^{n+1} ,  u^{n+1} - u^n \rangle
  =  \frac{\kappa}{2} ( \| \nabla_N u^{n+1} \|_2^2 - \| \nabla_N u^n \|_2^2 ) + \frac{\kappa}{2} \| \nabla_N (u^{n+1} - u^n ) \|_2^2  ,  \label{energy stab-4-2}
\end{eqnarray}
so that
\begin{eqnarray}
  \langle L_N u^{n+1} ,  u^{n+1} - u^n \rangle
  &=&  \frac{\varepsilon^2}{2} ( \| \Delta_N u^{n+1} \|_2^2 - \| \Delta_N u^n \|_2^2 )
  + \frac{\kappa}{2} ( \| \nabla_N u^{n+1} \|_2^2 - \| \nabla_N u^n \|_2^2 )  \nonumber
\\
  &&
  + \frac{\varepsilon^2}{2} \| \Delta_N (u^{n+1} - u^n ) \|_2^2
  + \frac{\kappa}{2} \| \nabla_N (u^{n+1} - u^n ) \|_2^2  .  \label{energy stab-4-3}
\end{eqnarray}
Finally, a combination of~\eqref{energy stab-3} and \eqref{energy stab-4-3} leads to~\eqref{energy stab-0}, due to the fact that $\kappa = \kappa_0 + \kappa^*$. This completes the proof of Lemma~\ref{lem: energy stab-0}.
\end{proof}

Moreover, a careful Fourier eigenvalue analysis enables one to derive the following estimate, which will be used in the energy stability analysis.

\begin{prop}  \label{prop: eigen est}
For any grid function $f$ with $\overline{f}=0$, the following inequality is valid:
\begin{eqnarray}
  ( \frac{\varepsilon^2}{2} + A \dt^2 ) \| \Delta_N f \|_2^2
   + \kappa^* \| \nabla_N f \|_2^2  +  \langle ( \frac{1}{\dt} {\cal G}_N - L_N ) f , f \rangle
  \ge  \gamma^{(0)} \| \nabla_N f \|_2^2 ,  \label{eigen est-0}
\end{eqnarray}
with $\gamma^{(0)}$ defined in~\eqref{constants-2}, provided that the coefficient $A$ satisfies
\begin{eqnarray}
  A \ge ( \gamma^{(0)} - \frac{\kappa}{4} )^4 \alpha_0^{-2}  \varepsilon^{-2} ,  \quad
  \mbox{with $\gamma^{(0)}$, $\alpha_0$ given by~\eqref{constants-2}, \eqref{constants-3} } .   \label{constants-4}
\end{eqnarray}
\end{prop}

\begin{proof}
For the grid function $f$ with discrete Fourier expansion~\eqref{Fourier-1} and the definition of the operator ${\cal G}$ in~\eqref{Fourier-2-1}, we see that
\begin{eqnarray}
  ( \frac{1}{\dt} {\cal G}_N - L_N ) f_{i,j} &=& \sum_{k,\ell=-K}^K  \Bigl(
   \frac{\Lambda_{k,\ell}}{1 - {\rm e}^{- \dt \Lambda_{k,\ell}} }  - \Lambda_{k,\ell} \Bigr)
  \hat{f}_{k,\ell} {\rm e}^{2 \pi i ( k x_i + \ell y_j)/L} \nonumber
\\
  &=&
  \sum_{k,\ell=-K}^K
   \frac{\Lambda_{k,\ell} {\rm e}^{- \dt \Lambda_{k,\ell}}  }{1 - {\rm e}^{- \dt \Lambda_{k,\ell}} }  \hat{f}_{k,\ell} {\rm e}^{2 \pi i ( k x_i + \ell y_j)/L} ,
     \label{eigen est-1}
\end{eqnarray}
with $\Lambda_{k,\ell}$ and $\lambda_{k,\ell}$ defined in~\eqref{spectral-coll-4-4}. In turn, an application of Parseval equality to~\eqref{eigen est-1} implies that
\begin{eqnarray}
 \langle ( \frac{1}{\dt} {\cal G}_N - L_N ) f , f \rangle
 = L^2 \sum_{k,\ell=-K}^K
   \frac{\Lambda_{k,\ell} {\rm e}^{- \dt \Lambda_{k,\ell}}  }{1 - {\rm e}^{- \dt \Lambda_{k,\ell}} }  | \hat{f}_{k,\ell} |^2 .    \label{eigen est-2-1}
\end{eqnarray}
Similarly, an application of Parseval equality to $\nabla_N f$ and $\Delta_N f$ gives
\begin{eqnarray}
  \| \nabla_N f \|_2^2 = L^2 \sum_{k,\ell=-K}^K
   \lambda_{k, \ell} | \hat{f}_{k,\ell} |^2  ,  \quad
   \| \Delta_N f \|_2^2 = L^2 \sum_{k,\ell=-K}^K
   \lambda_{k, \ell}^2 | \hat{f}_{k,\ell} |^2 .    \label{eigen est-2-2}
\end{eqnarray}
Therefore, the two sides of~\eqref{eigen est-0} become the following expansions, in terms of Fourier coefficients:
\begin{eqnarray}
  &&
  ( \frac{\varepsilon^2}{2} + A \dt^2 ) \| \Delta_N f \|_2^2
   + \kappa^* \| \nabla_N f \|_2^2  +  \langle ( \frac{1}{\dt} {\cal G}_N - L_N ) f , f \rangle
   = L^2 \sum_{k,\ell=-K}^K {\cal R}^{(1)}_{k, \ell} | \hat{f}_{k,\ell} |^2 ,
   \label{eigen est-2-3}
\\
  &&  \mbox{with} \quad
  {\cal R}^{(1)}_{k, \ell} =  \frac{\Lambda_{k,\ell} {\rm e}^{- \dt \Lambda_{k,\ell}}  }{1 - {\rm e}^{- \dt \Lambda_{k,\ell}} }  + ( \frac{\varepsilon^2}{2} + A \dt^2 ) \lambda_{k,\ell}^2 + \kappa^* \lambda_{k, \ell} ,    \label{eigen est-2-4}
\\
  &&
   \gamma^{(0)} \| \nabla_N f \|_2^2
   = L^2 \sum_{k,\ell=-K}^K {\cal R}^{(2)}_{k, \ell} | \hat{f}_{k,\ell} |^2  ,   \quad \mbox{with} \, \, \, {\cal R}^{(2)}_{k, \ell} = \gamma^{(0)} \lambda_{k, \ell} .
   \label{eigen est-2-5}
\end{eqnarray}

The rest work is focused on the comparison between ${\cal R}^{(1)}_{k, \ell}$ and ${\cal R}^{(2)}_{k, \ell}$, for any $k$, $\ell$. First, we observe that
\begin{eqnarray}
    \frac{\varepsilon^2}{2} \lambda_{k,\ell}^2 + \kappa^* \lambda_{k, \ell}
    \ge \frac12 \Lambda_{k, \ell} ,  \quad \mbox{since $\kappa^* \ge \frac12 \kappa$} .
\label{eigen est-3-1}
\end{eqnarray}
Consequently, if $\dt \Lambda_{k, \ell} \le \alpha_0$, (with $\alpha_0$ defined in~\eqref{constants-3}), the following inequality is valid:
\begin{eqnarray}
    {\cal R}^{(1)}_{k, \ell} &\ge&  \frac{\Lambda_{k,\ell} {\rm e}^{- \dt \Lambda_{k,\ell}}  }{1 - {\rm e}^{- \dt \Lambda_{k,\ell}} }  + \frac{\varepsilon^2}{2}  \lambda_{k,\ell}^2 + \kappa^* \lambda_{k, \ell}
    \ge \frac{\Lambda_{k,\ell} {\rm e}^{- \dt \Lambda_{k,\ell}}  }{1 - {\rm e}^{- \dt \Lambda_{k,\ell}} }  + \frac12 \Lambda_{k, \ell}  \nonumber
\\
  &\ge&
  \Lambda_{k,\ell}  \Bigl( \frac{ {\rm e}^{- \dt \Lambda_{k,\ell}}  }{1 - {\rm e}^{- \dt \Lambda_{k,\ell}} }  + \frac12 \Bigr)  \ge  \frac{\gamma^{(0)} }{\kappa}  \Lambda_{k,\ell}
  \nonumber
\\
  &\ge&
  \gamma^{(0)} \lambda_{k, \ell} = {\cal R}^{(2)}_{k, \ell} ,   \quad  \mbox{if $\dt \Lambda_{k, \ell} \le \alpha_0$} .    \label{eigen est-3}
\end{eqnarray}
On the other hand, if $\dt \Lambda_{k, \ell} > \alpha_0$, we have $\Lambda_{k, \ell} > \frac{\alpha_0}{\dt}$, so that
\begin{eqnarray}
    {\cal R}^{(1)}_{k, \ell} &\ge&  \frac{\varepsilon^2}{2}  \lambda_{k,\ell}^2 + \kappa^* \lambda_{k, \ell}  + A \dt^2 \lambda_{k, \ell}^2
    \ge \frac12 \Lambda_{k,\ell} + A \dt^2 \lambda_{k, \ell}^2
  = \frac14 \Lambda_{k,\ell} + \frac14 \Lambda_{k,\ell} + A \dt^2 \lambda_{k, \ell}^2  \nonumber
\\
  &\ge&
  \frac14 \Lambda_{k,\ell} + \frac{\alpha_0}{4 \dt} + A \dt^2 \lambda_{k, \ell}^2
  = \frac{\alpha_0}{4 \dt} + ( \frac{\varepsilon^2}{4} + A \dt^2 ) \lambda_{k, \ell}^2 + \frac{\kappa}{4} \lambda_{k, \ell} .  \label{eigen est-4-1}
\end{eqnarray}
Under the constraint~\eqref{constants-4}, the following quadratic inequality is valid
\begin{eqnarray}
  \frac{\varepsilon^2}{4} + A \dt^2
  \ge A^{1/2} \varepsilon \dt \ge ( \gamma^{(0)} - \frac{\kappa}{4} )^2 \alpha_0^{-1} \dt ,
  \label{eigen est-4-2}
\end{eqnarray}
which in turn results in
\begin{eqnarray}
    {\cal R}^{(1)}_{k, \ell} &\ge&   \frac{\alpha_0}{4 \dt} + ( \gamma^{(0)} - \frac{\kappa}{4} )^2 \alpha_0^{-1} \dt \lambda_{k, \ell}^2 + \frac{\kappa}{4} \lambda_{k, \ell} \nonumber
\\
  &\ge&
  ( \gamma^{(0)} - \frac{\kappa}{4} ) \lambda_{k, \ell}
  + \frac{\kappa}{4} \lambda_{k, \ell} = \gamma^{(0)} \lambda_{k, \ell}
  = {\cal R}^{(2)}_{k, \ell} ,  \quad \mbox{if $\dt \Lambda_{k, \ell} > \alpha_0$} .
  \label{eigen est-4-3}
\end{eqnarray}

Finally, a combination of~\eqref{eigen est-2-3}, \eqref{eigen est-2-5},
\eqref{eigen est-3} and \eqref{eigen est-4-3} yields the desired estimate~\eqref{eigen est-0}. This completes the proof of Proposition~\ref{prop: eigen est}.
\end{proof}




The energy stability of the proposed third order ETD-based scheme~\eqref{scheme-ETD-3rd-0}-\eqref{scheme-ETD-3rd-1} is stated in the following theorem, in a modified version.  

\begin{thm}  \label{thm:energy stab-ETDMs3}
  The numerical solution produced by the proposed scheme~\eqref{scheme-ETD-3rd-0}-\eqref{scheme-ETD-3rd-1} satisfies
\begin{eqnarray}
  \hspace{-0.35in} &&
  \tilde{E}_N (u^{n+1} , u^n, u^{n-1}) \le \tilde{E} (u^n, u^{n-1}, u^{n-2}) ,   \quad \mbox{with}  \nonumber
\\
  \hspace{-0.35in} &&
  \tilde{E}_N (u^{n+1} , u^n, u^{n-1}) = E_N (u^{n+1}) + \gamma_1^{(0)} \| \nabla_N (u^{n+1} - u^n ) \|_2^2 + \gamma^{(0)}_3 \| \nabla_N (u^n - u^{n-1}) \|_2^2  ,  \label{scheme-ETD-3rd-stability-0}
\end{eqnarray}
for any $\dt > 0$, with $\gamma_j^{(0)}$ ($j=1, 2, 3$) defined in~\eqref{constants-1}, provided that~\eqref{constants-4} is satisfied.
\end{thm}

\begin{proof}
  Applying an operator $(\dt \phi_0 (L_N) )^{-1}$ to~\eqref{scheme-ETD-3rd-0}, and making use of the operators defined in~\eqref{scheme-ETD-3rd-1}, \eqref{operator-2}, \eqref{operator-3}, we get a rewritten form of the nonlinear term $f_N (u^n)$:
\begin{eqnarray}
  f_N (u^n) &=& - L_N (I - {\rm e}^{-\dt L_N} )^{-1} (u^{n+1} - u^n) - L_N u^n - A \dt^2 \Delta_N^2 (u^{n+1} - u^n )  \nonumber
\\
  &&
  - G^{(1)}_N  ( \frac32 f_N (u^n) - 2 f_N (u^{n-1}) + \frac12 f_N (u^{n-2}) )  \nonumber
\\
   &&
- G^{(2)}_N \left( \frac12 f_N (u^n) - f_N (u^{n-1}) + \frac12 f_N (u^{n-2}) \right) .
  \label{scheme-ETD-3rd-stability-1}
\end{eqnarray}
With an application of the preliminary energy inequality~\eqref{energy stab-0} in Lemma~\ref{lem: energy stab-0}, we have
\begin{eqnarray}
  &&
  E_N (u^{n+1}) - E_N (u^n) + \frac{\varepsilon^2}{2} \| \Delta_N (u^{n+1} - u^n) \|_2^2
   + \kappa^* \| \nabla_N (u^{n+1} - u^n ) \|_2^2  \nonumber
 \\
   &\le& \langle L_N u^{n+1} + f_N (u^n), u^{n+1} - u^n \rangle \nonumber
\\
  &=&
  \langle - L_N (I - {\rm e}^{-\dt L_N} )^{-1} (u^{n+1} - u^n) + L_N (u^{n+1} - u^n) , u^{n+1} - u^n \rangle  \nonumber
\\
  &&
  - A \dt^2 \langle \Delta_N^2 (u^{n+1} - u^n ) , u^{n+1} - u^n \rangle  \nonumber
\\
  &&
  - \langle G^{(1)}_N  ( \frac32 f_N (u^n) - 2 f_N (u^{n-1}) + \frac12 f_N (u^{n-2}) ) , u^{n+1} - u^n \rangle   \nonumber
\\
  &&
  - \langle G^{(2)}_N  ( \frac12 f_N (u^n) - f_N (u^{n-1}) + \frac12 f_N (u^{n-2}) ) , u^{n+1} - u^n \rangle  .  \label{scheme-ETD-3rd-stability-2}
\end{eqnarray}
The first term on the right hand side turns out to be
\begin{eqnarray}
  &&
  \langle - L_N (I - {\rm e}^{-\dt L_N} )^{-1} (u^{n+1} - u^n) + L_N (u^{n+1} - u^n) , u^{n+1} - u^n \rangle   \nonumber
\\
  &=&
  - \frac{1}{\dt} \langle {\cal G}_N (u^{n+1} - u^n) , u^{n+1} - u^n \rangle + \langle L_N (u^{n+1} - u^n) , u^{n+1} - u^n \rangle  .  \label{scheme-ETD-3rd-stability-3-1}
\end{eqnarray}
The second term on the right hand side becomes
\begin{eqnarray}
   - A \dt^2 \langle \Delta_N^2 (u^{n+1} - u^n ) , u^{n+1} - u^n \rangle
   = - A \dt^2 \| \Delta_N^2 (u^{n+1} - u^n ) \|_2^2 . \label{scheme-ETD-3rd-stability-3-2}
\end{eqnarray}

For the third term on the right hand side of~\eqref{scheme-ETD-3rd-stability-2}, we begin with the following rewritten form:
\begin{eqnarray}
  &&
  \frac32 f_N (u^n) - 2 f_N (u^{n-1}) + \frac12 f_N (u^{n-2})  \nonumber
\\
  &=&
    \frac32 \nabla_N \cdot ( g_N (u^n) - g_N (u^{n-1}) )
    - \frac12 \nabla_N \cdot ( g_N (u^{n-1}) - g_N (u^{n-2}) ) .
    \label{scheme-ETD-3rd-stability-4-1}
\end{eqnarray}
Meanwhile, a careful calculation
\begin{eqnarray}
   g_N (u^n) - g_N (u^{n-1}) = \frac{\nabla_N u^n}{1+|\nabla_N u^n |^2} - \frac{\nabla_N u^{n-1}}{1+|\nabla_N u^{n-1}|^2}  + \kappa \nabla_N (u^n - u^{n-1})   \label{scheme-ETD-3rd-stability-4-2}
\end{eqnarray}
implies that
\begin{eqnarray}
  \| g_N (u^n) - g_N (u^{n-1}) \|_2
  \le (1 + \kappa ) \| \nabla_N (u^n - u^{n-1}) \|_2 , \label{scheme-ETD-3rd-stability-4-3}
\end{eqnarray}
with an application of Lemma~\ref{lem:vector inequality}. Then we apply summation by part formula and obtain
\begin{eqnarray}
  &&
  - \frac32 \langle G^{(1)}_N  \nabla_N \cdot ( g_N (u^n) - g_N (u^{n-1}) ) , u^{n+1} - u^n \rangle  \nonumber
\\
  &=&
   \frac32 \langle G^{(1)}_N  ( g_N (u^n) - g_N (u^{n-1}) ) , \nabla_N ( u^{n+1} - u^n ) \rangle  \nonumber
\\
  &\le&
  \frac32 \| G^{(1)}_N ( g_N (u^n) - g_N (u^{n-1}) ) \|_2 \cdot \| \nabla_N ( u^{n+1} - u^n ) \|_2   \nonumber
\\
  &\le&
  \frac32 C_4 \| g_N (u^n) - g_N (u^{n-1}) \|_2 \cdot \| \nabla_N ( u^{n+1} - u^n ) \|_2  \nonumber
\\
  &\le&
  \frac32 C_4 (1 + \kappa) \| \nabla_N (u^n - u^{n-1}) \|_2 \cdot \| \nabla_N ( u^{n+1} - u^n ) \|_2   \nonumber
\\
  &\le&
  \frac34 C_4 (1 + \kappa) ( \| \nabla_N (u^n - u^{n-1}) \|_2^2  + \| \nabla_N ( u^{n+1} - u^n ) \|_2^2 ) ,  \label{scheme-ETD-3rd-stability-4-4}
\end{eqnarray}
in which inequality~\eqref{prop-1-0-4} has been applied in the third step. Similar estimate could be derived in the same fashion:
\begin{eqnarray}
  &&
  \frac12 \langle G^{(1)}_N  \nabla_N \cdot ( g_N (u^{n-1}) - g_N (u^{n-2}) ) , u^{n+1} - u^n \rangle  \nonumber
\\
  &\le&
  \frac14 C_4 (1 + \kappa) ( \| \nabla_N (u^{n-1} - u^{n-2}) \|_2^2  + \| \nabla_N ( u^{n+1} - u^n ) \|_2^2 ) .  \label{scheme-ETD-3rd-stability-4-5}
\end{eqnarray}
In turn, a combination of~\eqref{scheme-ETD-3rd-stability-4-1}, \eqref{scheme-ETD-3rd-stability-4-4} and \eqref{scheme-ETD-3rd-stability-4-5} yields
\begin{eqnarray}
  \hspace{-0.35in} &&
  - \langle G^{(1)}_N  ( \frac32 f_N (u^n) - 2 f_N (u^{n-1}) + \frac12 f_N (u^{n-2}) ), u^{n+1} - u^n \rangle  \nonumber
\\
  \hspace{-0.35in} &\le&
  C_4 (1 + \kappa) \Bigl( \| \nabla_N (u^{n+1} - u^n) \|_2^2 + \frac34
  \| \nabla_N (u^n - u^{n-1}) \|_2^2 +  \frac14 \| \nabla_N (u^{n-1} - u^{n-2}) \|_2^2 \Bigr) .  \label{scheme-ETD-3rd-stability-4-6}
\end{eqnarray}

The last term on the right hand side of~\eqref{scheme-ETD-3rd-stability-2} could be analyzed in a similar form; the details are left to interested readers:
\begin{eqnarray}
  \hspace{-0.35in} &&
  - \langle G^{(2)}_N  ( \frac12 f_N (u^n) - f_N (u^{n-1}) + \frac12 f_N (u^{n-2}) ), u^{n+1} - u^n \rangle  \nonumber
\\
  \hspace{-0.35in} &\le&
  C_5 (1 + \kappa) \Bigl( \frac12 \| \nabla_N (u^{n+1} - u^n) \|_2^2 + \frac14 \| \nabla_N (u^n - u^{n-1}) \|_2^2 + \frac14 \| \nabla_N (u^{n-1} - u^{n-2}) \|_2^2 \Bigr).  \label{scheme-ETD-3rd-stability-5-1}
\end{eqnarray}
Its combination with~\eqref{scheme-ETD-3rd-stability-4-6} leads to
\begin{eqnarray}
    \hspace{-0.35in} &&
 -  \langle G^{(1)}_N  ( \frac32 f_N (u^n) - 2 f_N (u^{n-1}) + \frac12 f_N (u^{n-2}) ), u^{n+1} - u^n \rangle  \nonumber
\\
  \hspace{-0.35in} &&
  - \langle G^{(2)}_N  ( \frac12 f_N (u^n) - f_N (u^{n-1}) + \frac12 f_N (u^{n-2}) ), u^{n+1} - u^n \rangle  \nonumber
\\
  \hspace{-0.35in} &\le&
  \gamma^{(0}_1  \| \nabla_N (u^{n+1} - u^n) \|_2^2 + \gamma^{(0}_2 \| \nabla_N (u^n - u^{n-1}) \|_2^2 + \gamma^{(0)}_3 \| \nabla_N (u^{n-1} - u^{n-2}) \|_2^2 ,   \label{scheme-ETD-3rd-stability-5-2}
\end{eqnarray}
in which $\gamma^{(0)}_j$ ($j=1, 2, 3$) has been given by~\eqref{constants-1}, and we have used the fact that $C_4 = C_5 = \frac{1}{1 - {\rm e}^2}$.

  Finally, a substitution of~\eqref{scheme-ETD-3rd-stability-3-1}, \eqref{scheme-ETD-3rd-stability-3-2},  and \eqref{scheme-ETD-3rd-stability-5-2} into~\eqref{scheme-ETD-3rd-stability-2} results in \eqref{scheme-ETD-3rd-stability-0}. Notice that we have made use of the following inequality
\begin{eqnarray}
  &&
     ( \frac{\varepsilon^2}{2} + A \dt^2) \| \Delta_N (u^{n+1} - u^n) \|_2^2
   + \kappa^* \| \nabla_N (u^{n+1} - u^n ) \|_2^2 \nonumber
\\
  &+&
  \langle ( \frac{1}{\dt} {\cal G}_N - L_N ) (u^{n+1} - u^n) , u^{n+1} - u^n \rangle
  \ge \gamma^{(0)} \| \nabla_N (u^{n+1} - u^n) \|_2^2 ,
\end{eqnarray}
which comes from an application of Proposition~\ref{prop: eigen est}, provided that~\eqref{constants-4} is satisfied. This completes the proof of Theorem~\ref{thm:energy stab-ETDMs3}.
\end{proof}

\begin{cor}  \label{cor:energy bound-ETDMs3}
  The numerical solution~\eqref{scheme-ETD-3rd-0}-\eqref{scheme-ETD-3rd-1}, we have
\begin{eqnarray}
  E_N (u^k) \le E_N (u^0) + \gamma_1^{(0)} \| \nabla_N (u^0 - u^{-1} ) \|_2^2 + \gamma^{(0)}_3 \| \nabla_N (u^{-1} - u^{-2}) \|_2^2 := \tilde{C}_0 , \quad \forall k \ge 0 ,
  \label{energy bound-0}
\end{eqnarray}
provided that~\eqref{constants-4} is satisfied.
\end{cor}

\begin{proof}
By the modified energy inequality~\eqref{scheme-ETD-3rd-stability-0}, the following induction analysis could be performed:
\begin{eqnarray}
  \hspace{-0.35in} &&
  E_N (u^k ) \le \tilde{E}_N (u^k , u^{k-1}, u^{k-2}) \le ... \le  \tilde{E}_N (u^0 , u^{-1}, u^{-2}) := \tilde{C}_0 ,  \quad \forall k \ge 0 .
  \label{energy bound-1}
\end{eqnarray}
\end{proof}

\begin{rem}
For the first order ETD scheme~\eqref{scheme-ETD-1st}, the energy stability has been proved in~\cite{Ju17}, namely, $E_N (u^{n+1}) \le E_N (u^n)$, for any $\dt >0$, provided that $\kappa \ge \frac18$. A similar analysis could also be found in an earlier work~\cite{chen12}.

For the original version of the second order accurate ETD-based scheme~\eqref{scheme-ETD-2nd}, such an energy estimate could hardly be theoretically justified, due to its multi-step nature. Instead, an alternate energy inequality has been derived in~\cite{Ju17} for~\eqref{scheme-ETD-2nd}:
 \begin{eqnarray}
  E_N (u^{n+1}) \le E_N (u^n) + \frac{1+\kappa}{2} ( \| \nabla_N (u^{n+1} - u^n) \|_2^2 + \| \nabla_N (u^n - u^{n-1}) \|_2^2 )  .  \label{scheme-ETD-2nd-stability-0}
\end{eqnarray}
However, because of the two positive terms on the right hand side, a uniform in time bound of the original energy functional is not theoretically available from such an energy inequality.

For the third order ETD-based scheme~\eqref{scheme-ETD-3rd-Ju17} reported in~\cite{Ju17}, a bound estimate for the original energy functional has not been available, either. Our energy analysis has revealed that, an artificial regularization term is needed to theoretically justify the energy stability for a higher-order ETD-based schemes, as demonstrated in the proof of Theorem~\ref{thm:energy stab-ETDMs3}.
\end{rem}

\begin{rem}
The requirement~\eqref{constants-4} for the parameter $A$ indicates an order of $A=O (\varepsilon^{-2})$, since $\gamma^{(0)} = O(1)$, $\alpha_0 = O(1)$. Such a requirement is based on a subtle fact that, an extra stability estimate from the surface diffusion term has to be used to balance the stability loss coming from the multi-step explicit treatment of the nonlinear terms, and the surface diffusion coefficient is given by $\varepsilon^2$ in the physical parameter.

On the other hand, such a parameter order $A= O (\varepsilon^{-2})$ is only used for the theoretical justification of the energy stability. In the practical computations, an extensive choice of $A= O (1)$ has never led to any energy stability loss for the proposed third order scheme.
\end{rem}

\begin{rem}
There have been other approaches to obtain the desired stability property for the no-slope-selection thin film model. In another recent work~\cite{chen18b}, an artificial stabilizing term $A \dt^2\frac{\partial\Delta^2 u}{\partial t}$ is added to the second order accurate ETD-based scheme, which was exactly integrated over the time interval $(t^n, t^{n+1})$. Furthermore, a careful analysis justifies the numerical stability, with the parameter $A$ of an order $A = O (1)$. A similar result has also been reported in~\cite{LiW18}, in which a second order modified BDF approximation is used, and an artificial regularization with a parameter $A \ge \frac{25}{16}$ is sufficient to ensure the energy stability.

The primary reason for the difference in the order of the artificial parameter $A$ between the second and third order numerical schemes is based on the following fact: for the second order scheme, the artificial regularization, with magnitude $O(\dt^2)$, and the temporal discretization terms are sufficient to theoretically justify the energy stability; while for the third order scheme, these two terms are not sufficient to ensure the numerical stability, since the artificial regularization term has to be in the order of $O(\dt^3)$ to keep the third order temporal accuracy. As a result, the stability estimate for surface diffusion term has to be involved, so that an coefficient $A =O (\varepsilon^{-2})$ is needed to pass through the stability analysis.
\end{rem}

\begin{rem}
For the epitaxial thin film growth model with slope selection~\eqref{energy-SS}, there have been some energy stability analysis works for the linear stabilization schemes. In more details, the artificial regularization parameter is required to be of order $O(\varepsilon^{-2} | \ln \varepsilon |)$ for the first order scheme~\cite{LiD2016a}, while such a parameter becomes of order $O (\varepsilon^{-m_0})$ (with $m_0 \ge 10$) for the second order scheme~\cite{LiD2017, LiD2017b, SongH2017}, in which the authors used the Cahn-Hilliard model to illustrate the analysis techniques. For the third and higher order numerical schemes, there has been no theoretical justification of the energy stability analysis for the slope-selection model.

The reason why we are able to obtain such a sharp theoretical result is associated with the subtle fact that: the nonlinear term in the NSS equation~\eqref{equation-NSS} has automatically bounded higher order derivatives in the $L^\infty$ norm, which enables one to derive an energy estimate with a much reduced artificial regularization for higher order numerical schemes.
\end{rem}

\section{The convergence analysis for the third order ETD-based scheme}
\label{sec-convergence}

The global existence of weak solution, strong solution and smooth solution for the NSS equation (\ref{equation-NSS}) has been established in~\cite{libo03}. In more details, a global in time estimate of $L^\infty (0,T; H^m) \cap L^2 (0,T; H^{m+2})$ for the phase variable was proved, assuming initial data in $H^m$, for any $m \ge 2$. Therefore, with an initial data with sufficient regularity, we could assume that the exact solution has regularity of class $\mathcal{R}$:
		\begin{equation}
			u_{e} \in \mathcal{R} := H^4 (0,T; C^0) \cap H^1 (0,T; H^4) \cap H^3 (0,T; H^{m+2}) \cap L^\infty (0,T; H^{m+4}).
			\label{assumption:regularity.1}
		\end{equation}


Define $U_N (\, \cdot \, ,t) := {\cal P}_N u_e (\, \cdot \, ,t)$, the (spatial) Fourier projection of the exact solution into ${\cal B}^K$, the space of trigonometric polynomials of degree to and including  $K$.  The following projection approximation is standard: if $\phi_e \in L^\infty(0,T;H^\ell_{\rm per}(\Omega))$, for some $\ell\in\mathbb{N}$,
	\begin{equation}
\| U_N - u_e \|_{L^\infty(0,T;H^m)}
   \le C h^{\ell-k} \| u_e \|_{L^\infty(0,T;H^\ell)},  \quad \forall \ 0 \le k \le \ell .
	\label{projection-est-0}
	\end{equation}
By $U_N^m$, $U^m$ we denote $U_N(\, \cdot \, , t^m)$ and $u_e (\, \cdot \, , t^m)$, respectively, with $t^m = m\cdot \dt$. Since $U_N \in {\cal P}_K$, the mass conservative property is available at the discrete level:
	\begin{equation}
\overline{U_N^m} = \frac{1}{|\Omega|}\int_\Omega \, U_N ( \cdot, t_m) \, d {\bf x} = \frac{1}{|\Omega|}\int_\Omega \, U_N ( \cdot, t_{m-1}) \, d {\bf x} = \overline{U_N^{m-1}} ,  \quad \forall \ m \in\mathbb{N}.
	\label{mass conserv-1}
	\end{equation}
On the other hand, the solution of the proposed scheme~\eqref{scheme-ETD-3rd-0} is also mass conservative at the discrete level:
	\begin{equation}
\overline{u^m} = \overline{u^{m-1}} ,  \quad \forall \ m \in \mathbb{N} .
	\label{mass conserv-2}
	\end{equation}
Meanwhile, we denote $U^m$ as the interpolation values of $U_N$ at discrete grid points at time instant $t^m$: $U_{i,j}^m :=  U_N (x_i, y_j, t^m)$.
As indicated before, we use the mass conservative projection for the initial data:  
	\begin{equation}
u^0_{i,j} = U^0_{i,j} := U_N (x_i, y_j, t=0) .
	\label{initial data-0}
	\end{equation}	
The error grid function is defined as
	\begin{equation}
e^m := U^m - u^m ,  \quad \forall \ m \in \left\{ 0 ,1 , 2, 3, \cdots \right\} .
	\label{error function-1}
	\end{equation}
Therefore, it follows that  $\overline{e^m} =0$, for any $m \in \left\{ 0 ,1 , 2, 3, \cdots \right\}$. 

For the proposed third order accurate scheme~\eqref{scheme-ETD-3rd-0}-\eqref{scheme-ETD-3rd-1}, the convergence result is stated below.

    \begin{thm}
    	\label{thm:convergence}
    	Given initial data $U_N^{0}$, $U_N^{-1}$, $U_N^{-2} \in C_{\rm per}^{m+4} (\overline{\Omega})$, with periodic boundary conditions, suppose the unique solution for the the NSS equation (\ref{equation-NSS}) is of regularity class $\mathcal{R}$. Then, provided $\dt$ and $h$ are sufficiently small, 
for all positive integers $\ell$, such that $\dt \cdot \ell \le T$, we have
    	\begin{eqnarray}
       &&	
    	\| \nabla_N e^\ell \|_2 +  \Bigl( \varepsilon^2 \dt   \sum_{m=1}^{\ell} \| \nabla_N \Delta_N e^m \|_2^2 \Bigr)^{1/2}
    		\le C ( \dt^3 + h^m ),   \label{convergence-0-2}
    	\end{eqnarray}
    	where $C>0$ is independent of $\dt$ and $h$.
    \end{thm}

\subsection{The error evolutionary equation}

  For the Fourier projection solution $U_N$ and its interpolation $U$, a careful consistency analysis implies that
\begin{eqnarray}
  &&
  U^{n+1} = {\rm e}^{- \dt L_N} U^n - A \dt^3 \phi_0 (L_N) \Delta_N^2 ( U^{n+1} - U^n ) - \dt \phi_0 (L_N) f_N (U^n)  \nonumber
\\
  &&
  - \dt \phi_1 (L_N) \left( \frac32 f_N (U^n) - 2 f_N (U^{n-1}) + \frac12 f_N (U^{n-2}) \right)  \nonumber
\\
  &&  \qquad
  - \dt \phi_2 (L_N) \left( \frac12 f_N (U^n) - f_N (U^{n-1}) + \frac12 f_N (U^{n-2}) \right) + \dt \tau^n ,
  \label{ETD-3rd-consistency-1}
\end{eqnarray}
with $\| \tau^n \|_{H_h^3} \le C (\dt^3 + h^m)$. In turn, subtracting the numerical scheme \eqref{scheme-ETD-3rd-0} from the consistency estimate \eqref{ETD-3rd-consistency-1} yields
\begin{eqnarray}
  e^{n+1} =& {\rm e}^{- \dt L_N} e^n - A \dt^3 \phi_0 (L_N) \Delta_N^2 ( e^{n+1} - e^n ) - \dt \phi_0 (L_N) \tilde{f}_N (U^n, u^n)  \nonumber
\\
  &
  - \dt \phi_1 (L_N) \Bigl( \frac32 \tilde{f}_N (U^n, u^n) - 2 \tilde{f}_N (U^{n-1}, u^{n-1}) + \frac12 \tilde{f}_N (U^{n-2}, u^{n-2}) \Bigr)  \nonumber
\\
  &
  - \dt \phi_2 (L_N) \Bigl( \frac12 \tilde{f}_N (U^n, u^n) - \tilde{f}_N (U^{n-1}, u^{n-1})
  + \frac12 \tilde{f}_N (U^{n-2}, u^{n-2}) \Bigr)  + \dt \tau^n ,  \label{ETD-3rd-consistency-2}
\end{eqnarray}
with $\tilde{f}_N (U^k, u^k) = f_N (U^k) - f_N (u^k)$, $k \ge 0$.

On the other hand, the current form~\eqref{ETD-3rd-consistency-2} for the numerical error evolution has not revealed a clear interaction between the linear and nonlinear terms. Instead, if we denote $e^{n+1,*}= {\rm e}^{- \dt L_N} e^n$, the error equation \eqref{ETD-3rd-consistency-2} could be rewritten as the following two-stage system, so that the corresponding error analysis could be carried out in a more convenient way:
\begin{eqnarray}
  \frac{e^{n+1,*} - e^n}{\dt} &=& - L_N \phi_0 (L_N)  e^n ,  \label{ETD-3rd-consistency-3-1}
\\
  \frac{e^{n+1} - e^{n+1,*}}{\dt} &=&  - A \dt^2 \phi_0 (L_N) \Delta_N^2 ( e^{n+1} - e^n ) + \tau^n - \phi_0 (L_N) \tilde{f}_N (U^n, u^n)  \nonumber
\\
  &-&
  \phi_1 (L_N) \Bigl( \frac32 \tilde{f}_N (U^n, u^n) - 2 \tilde{f}_N (U^{n-1}, u^{n-1})
  + \frac12 \tilde{f}_N (U^{n-2}, u^{n-2}) \Bigr)  \nonumber
\\
  &-&
  \phi_2 (L_N) \Bigl( \frac12 \tilde{f}_N (U^n, u^n) - \tilde{f}_N (U^{n-1}, u^{n-1}) + \frac12 \tilde{f}_N (U^{n-2}, u^{n-2}) \Bigr) .  \label{ETD-3rd-consistency-3-2}
\end{eqnarray}

\subsection{Some preliminary nonlinear error inequalities}

The following estimate for the nonlinear error term will be needed in later analysis; the detailed proof could be found in Appendix~\ref{proof:Prop 2}.

\begin{prop} \label{prop:prop 2}
Under the assumption that
\begin{equation}
  \| \nabla_N \Delta_N e^k \|_2 \le 1 ,
\label{a priori-0}
\end{equation}
the following inequality is available:
\begin{eqnarray}
  \nrm{ \nabla_N \left( f_N (U^k) - f_N (u^k) \right) }_2  \le C_0^{(2)} \| \nabla_N \Delta_N e^k \|_2 .
  \label{prop 2-1}
\end{eqnarray}
\end{prop}

\subsection{The $\ell^\infty (0,T; H_h^1) \cap \ell^2 (0,T; H_h^3)$ error estimate}

Now we get back to the error evolutionary system (\ref{ETD-3rd-consistency-3-1})-(\ref{ETD-3rd-consistency-3-2}). By applying the linear operator ${\cal G}_N = (\phi_0 (L_N) )^{-1} $ to both equations, we get
\begin{eqnarray}
  &&
  \frac{{\cal G}_N ( e^{n+1,*} - e^n )}{\dt} = - L_N  e^n ,  \label{ETD-3rd-consistency-4-1}
\\
  &&
  \frac{{\cal G}_N ( e^{n+1} - e^{n+1,*} ) }{\dt} =  - A \dt^2 \Delta_N^2 ( e^{n+1} - e^n ) +  {\cal G}_N \tau^n - \tilde{f}_N (U^n, u^n)  \nonumber
\\
  &&  \qquad
   - G^{(1)}_N  \Bigl( \frac32 \tilde{f}_N (U^n, u^n) - 2 \tilde{f}_N (U^{n-1}, u^{n-1})
  + \frac12 \tilde{f}_N (U^{n-2}, u^{n-2}) \Bigr)  \nonumber
\\
  &&  \qquad
  - G^{(2)}_N \Bigl( \frac12 \tilde{f}_N (U^n, u^n) - \tilde{f}_N (U^{n-1}, u^{n-1}) + \frac12 \tilde{f}_N (U^{n-2}, u^{n-2}) \Bigr) .  \label{ETD-3rd-consistency-4-2}
\end{eqnarray}
Since the nonlinear error inequality~\eqref{prop 2-1} is based on the a-priori assumption~\eqref{a priori-0}, we have to make such an assumption for any $k \le n$. This assumption will be recovered by the convergence estimate at the next time step.

Taking a discrete $L^2$ inner product with (\ref{ETD-3rd-consistency-4-1}) by $- \Delta_N (e^{n+1,*} + e^n)$ leads to
\begin{eqnarray}
   \left\langle {\cal G}_N ( e^{n+1,*} - e^n ) , - \Delta_N (e^{n+1,*} + e^n) \right\rangle
   + \dt \left\langle L_N  e^n , - \Delta_N (e^{n+1,*} + e^n) \right\rangle = 0 .
   \label{convergence-H1-1}
\end{eqnarray}
The first term could be analyzed with the help of the summation by parts identity \eqref{Fourier-4}:
\begin{eqnarray}
  &&
   \left\langle {\cal G}_N ( e^{n+1,*} - e^n ) , - \Delta_N (e^{n+1,*} + e^n) \right\rangle
   = \left\langle \nabla_N {\cal G}_N ( e^{n+1,*} - e^n ) , \nabla_N (e^{n+1,*} + e^n) \right\rangle \nonumber
\\
  &=&
    \left\langle \nabla_N {\cal G}^{(0)}_N ( e^{n+1,*} - e^n ) , \nabla_N {\cal G}^{(0)}_N (e^{n+1,*} + e^n) \right\rangle
  = \|  \nabla_N {\cal G}^{(0)}_N e^{n+1,*} \|_2^2
   - \|  \nabla_N {\cal G}^{(0)}_N e^n \|_2^2 .  \label{convergence-H1-2}
\end{eqnarray}
For the second term appearing in (\ref{convergence-H1-1}), we begin with the following observation:
\begin{eqnarray}
  &&
     \left\langle L_N  e^n , - \Delta_N e^{n+1,*}  \right\rangle
     = \left\langle L_N  e^n , - \Delta_N {\rm e}^{- L_N \dt} e^n  \right\rangle \ge 0 ,
     \label{convergence-H1-3}
\end{eqnarray}
in which the last step is based the second inequality of~\eqref{prop-1-0-1-2}. The other part could be analyzed in a more straightforward way:
\begin{eqnarray}
  \left\langle L_N  e^n , - \Delta_N e^n  \right\rangle
     = \left\langle ( \varepsilon^2 \Delta_N^2 - \kappa \Delta_N ) e^n , - \Delta_N e^n  \right\rangle  = \varepsilon^2 \| \nabla_N \Delta_N e^n \|_2^2 + \kappa \| \Delta_N  e^n \|_2^2 .  \label{convergence-H1-4}
\end{eqnarray}
In turn, a combination of (\ref{convergence-H1-1})-(\ref{convergence-H1-4}) results in
\begin{eqnarray}
  \|  \nabla_N {\cal G}^{(0)}_N e^{n+1,*} \|_2^2
   - \| \nabla_N {\cal G}^{(0)}_N e^n \|_2^2
   +  \varepsilon^2 \dt \| \nabla_N \Delta_N e^n \|_2^2 + \kappa \dt \| \Delta_N  e^n \|_2^2
   \le 0 . \label{convergence-H1-5}
\end{eqnarray}

Taking a discrete $L^2$ inner product with (\ref{ETD-3rd-consistency-4-2}) by $- 2 \Delta_N e^{n+1}$ yields
\begin{eqnarray}
  &&
   \left\langle {\cal G}_N ( e^{n+1} - e^{n+1,*} ) , - 2 \Delta_N e^{n+1} \right\rangle  \nonumber
\\
  &=&
    2 A \dt^3 \left\langle \Delta_N^2 (e^{n+1} - e^n) ,  \Delta_N e^{n+1}  \right\rangle
    + 2 \dt \left\langle {\cal G}_N \tau^n ,  - \Delta_N e^{n+1} \right\rangle
     + 2 \dt \left\langle \tilde{f}_N (U^n, u^n) , \Delta_N e^{n+1} \right\rangle \nonumber
\\
  &&
  + \dt \left\langle G_N^{(1)} ( 3 \tilde{f}_N (U^n, u^n) - 4 \tilde{f}_N (U^{n-1}, u^{n-1})
  + \tilde{f}_N (U^{n-2}, u^{n-2}) ) , \Delta_N e^{n+1} \right\rangle  \nonumber
\\
  &&
  + \dt \left\langle G_N^{(2)} ( \tilde{f}_N (U^n, u^n) - 2 \tilde{f}_N (U^{n-1}, u^{n-1})
  + \tilde{f}_N (U^{n-2}, u^{n-2}) ) , \Delta_N e^{n+1} \right\rangle .
   \label{convergence-H1-6}
\end{eqnarray}
The term on the left hand side could be analyzed in a similar way as (\ref{convergence-H1-2}):
\begin{eqnarray}
   \left\langle {\cal G}_N ( e^{n+1} - e^{n+1,*} ) , - 2 \Delta_N e^{n+1} \right\rangle
  &=&
    2 \left\langle \nabla_N {\cal G}^{(0)}_N ( e^{n+1} - e^{n+1,*} ) , \nabla_N {\cal G}^{(0)}_N e^{n+1} \right\rangle   \nonumber
\\
  &\ge&
  \|  \nabla_N {\cal G}^{(0)}_N e^{n+1} \|_2^2
   - \| \nabla_N {\cal G}^{(0)}_N e^{n+1,*} \|_2^2 .  \label{convergence-H1-7}
\end{eqnarray}
The first term on the right hand side turns out to be
\begin{eqnarray}
  2 \left\langle \Delta_N^2 (e^{n+1} - e^n) ,  \Delta_N e^{n+1}  \right\rangle
  &=& -2  \left\langle \nabla_N \Delta_N (e^{n+1} - e^n) ,  \nabla_N \Delta_N e^{n+1}  \right\rangle  \nonumber
\\
  &\le&
  - ( \|  \nabla_N \Delta_N e^{n+1} \|_2^2 - \| \nabla_N \Delta_N e^n \|_2^2 ) .
  \label{convergence-H1-8-1}
\end{eqnarray}
The bound for the truncation error term could be obtained as follows:
\begin{eqnarray}
   - \left\langle {\cal G}_N \tau^n ,  \Delta_N e^{n+1} \right\rangle
   &=&  \left\langle \nabla_N {\cal G}^{(0)}_N \tau^n ,  \nabla_N {\cal G}^{(0)}_N e^{n+1} \right\rangle  \nonumber
\\
  &\le&
   \frac12 (  \| \nabla_N {\cal G}^{(0)}_N \tau^n \|_2^2 + \| \nabla_N {\cal G}^{(0)}_N e^{n+1} \|_2^2 ) .   \label{convergence-H1-8-2}
\end{eqnarray}
For the first nonlinear inner product term, the following estimate could be derived:
\begin{eqnarray}
  &&
  \langle \tilde{f}_N (U^n, u^n) , \Delta_N e^{n+1} \rangle
  =  - \langle \nabla_N \left( f_N (U^n) - f_N (u^n) \right) , \nabla_N e^{n+1} \rangle
  \nonumber
\\
  &\le&
  \nrm{ \nabla_N \left( f_N (U^n) - f_N (u^n) \right) }_2
  \cdot  \| \nabla_N e^{n+1} \|_2
  \le C_0^{(2)} \nrm{ \nabla_N \Delta_N e^n }_2
  \cdot \| \nabla_N {\cal G}^{(0)}_N e^{n+1} \|_2  \nonumber
\\
  &\le&
    \frac{1}{16} \varepsilon^2 \nrm{ \nabla_N \Delta_N e^n }_2^2
  +  4 (C_0^{(2)})^2 \varepsilon^{-2} \| \nabla_N {\cal G}^{(0)}_N e^{n+1} \|_2^2  ,
\label{convergence-H1-9}
\end{eqnarray}
in which Propositions \ref{prop:prop 1}, \ref{prop:prop 2} have been applied in the third step. The nonlinear inner product involving $G_N^{(1)}$ could be handled as follows:
\begin{eqnarray}
   &&
  3 \langle G_N^{(1)} \tilde{f}_N (U^n, u^n) , \Delta_N e^{n+1} \rangle
  =  - 3 \langle G_N^{(1)} \nabla_N \left( f_N (U^n) - f_N (u^n) \right) , \nabla_N e^{n+1} \rangle
  \nonumber
\\
  &\le&
  3 \nrm{ G_N^{(1)} \nabla_N \left( f_N (U^n) - f_N (u^n) \right) }_2
  \cdot  \| \nabla_N e^{n+1} \|_2  \nonumber
\\
  &\le&
   3 C_4 \nrm{ \nabla_N \left( f_N (U^n) - f_N (u^n) \right) }_2
  \cdot  \nrm{ \nabla_N e^{n+1} }_2
  \le 3 C_0^{(2)} C_4 \nrm{ \nabla_N \Delta_N e^n }_2
  \cdot \| \nabla_N {\cal G}^{(0)}_N e^{n+1} \|_2  \nonumber
\\
  &\le&
    \frac18 \varepsilon^2 \nrm{ \nabla_N \Delta_N e^n }_2^2
  +  18 (C_0^{(2)})^2 C_4^2 \varepsilon^{-2} \| \nabla_N {\cal G}^{(0)}_N e^{n+1} \|_2^2  ,  \label{convergence-H1-10-1}
\end{eqnarray}
with the inequality (\ref{prop-1-0-4}) applied in the third step. The other nonlinear inner product terms could be analyzed in a similar way, and the following results are available:
\begin{eqnarray}
  &&
  - 4 \langle G_N^{(1)} \tilde{f}_N (U^{n-1}, u^{n-1}) , \Delta_N e^{n+1} \rangle
   \nonumber
\\
  &\le&
  \frac18 \varepsilon^2 \nrm{ \nabla_N \Delta_N e^{n-1} }_2^2
  +  32 (C_0^{(2)})^2 C_4^2 \varepsilon^{-2} \| \nabla_N {\cal G}^{(0)}_N e^{n+1} \|_2^2  ,  \label{convergence-H1-10-2}
\\
  &&
   \langle G_N^{(1)} \tilde{f}_N (U^{n-2}, u^{n-2}) , - \Delta_N e^{n+1} \rangle \nonumber
\\
  &\le&
   \frac18 \varepsilon^2 \nrm{ \nabla_N \Delta_N e^{n-2} }_2^2
  +  2 (C_0^{(2)})^2 C_4^2 \varepsilon^{-2} \| \nabla_N {\cal G}^{(0)}_N e^{n+1} \|_2^2  ,  \label{convergence-H1-10-3}
\\
  &&
  \langle G_N^{(2)} \tilde{f}_N (U^n, u^n) , \Delta_N e^{n+1} \rangle
  \le
 \frac18 \varepsilon^2 \nrm{ \nabla_N \Delta_N e^n }_2^2
  +  2 (C_0^{(2)})^2 C_5^2 \varepsilon^{-2} \| \nabla_N {\cal G}^{(0)}_N e^{n+1} \|_2^2  ,  \label{convergence-H1-10-4}
\\
  &&
  -2 \langle G_N^{(2)} \tilde{f}_N (U^{n-1}, u^{n-1}) , \Delta_N e^{n+1} \rangle
   \nonumber
 \\
   &\le&
   \frac18 \varepsilon^2 \nrm{ \nabla_N \Delta_N e^{n-1} }_2^2
  +  8 (C_0^{(2)})^2 C_5^2 \varepsilon^{-2} \| \nabla_N {\cal G}^{(0)}_N e^{n+1} \|_2^2  .  \label{convergence-H1-10-5}
\\
  &&
   \langle G_N^{(2)} \tilde{f}_N (U^{n-1}, u^{n-1}) , \Delta_N e^{n+1} \rangle  \nonumber
\\
  &\le&
  \frac18 \varepsilon^2 \nrm{ \nabla_N \Delta_N e^{n-2} }_2^2
  +  2 (C_0^{(2)})^2 C_5^2 \varepsilon^{-2} \| \nabla_N {\cal G}^{(0)}_N e^{n+1} \|_2^2  .  \label{convergence-H1-10-6}
\end{eqnarray}
We notice that all these estimates have to be based on the a-priori assumption~\eqref{a priori-0}, combined with the nonlinear error inequality~\eqref{prop 2-1} in Proposition~\ref{prop:prop 2}. In turn, a substitution of (\ref{convergence-H1-7})-(\ref{convergence-H1-10-6}) into (\ref{convergence-H1-6}) results in
\begin{eqnarray}
  &&
    \|  \nabla_N {\cal G}^{(0)}_N e^{n+1} \|_2^2
   - \| \nabla_N {\cal G}^{(0)}_N e^{n+1,*} \|_2^2
   + A \dt^3 ( \|  \nabla_N \Delta_N e^{n+1} \|_2^2 - \| \nabla_N \Delta_N e^n \|_2^2 )      \nonumber
\\
   &\le&  \frac14 \varepsilon^2 \dt ( \nrm{ \nabla_N \Delta_N e^{n-1} }_2^2
   + \nrm{ \nabla_N \Delta_N e^{n-2} }_2^2 )
   - \frac38 \varepsilon^2 \dt \nrm{ \nabla_N \Delta_N e^n }_2^2
   +  \dt \| \nabla_N {\cal G}^{(0)}_N \tau^n \|_2^2  \nonumber
\\
  &&
  +  \left( 72 (C_0^{(2)})^2 (C_4^2 + C_5^2) \varepsilon^{-2} + 1 \right)
  \dt \| \nabla_N {\cal G}^{(0)}_N e^{n+1} \|_2^2 .  \label{convergence-H1-11}
\end{eqnarray}
Its combination with (\ref{convergence-H1-5}) yields
\begin{eqnarray}
  &&
    \| \nabla_N {\cal G}^{(0)}_N e^{n+1} \|_2^2
   - \|  \nabla_N {\cal G}^{(0)}_N e^n \|_2^2
   + A \dt^2 ( \|  \nabla_N \Delta_N e^{n+1} \|_2^2 - \| \nabla_N \Delta_N e^n \|_2^2 )
   \nonumber
 \\
   &&
   +  \frac58 \varepsilon^2 \dt \nrm{ \nabla_N \Delta_N e^n }_2^2  \nonumber
\\
  &\le&
  \frac14 \varepsilon^2 \dt ( \| \nabla_N \Delta_N e^{n-1} \|_2^2
   + \| \nabla_N \Delta_N e^{n-2} \|_2^2 )
   + \dt \| \nabla_N {\cal G}^{(0)}_N \tau^n \|_2^2   \nonumber
\\
  &&
  +  \left( 72 (C_0^{(2)})^2 (C_4^2 + C_5^2)  \varepsilon^{-2} + 1 \right)
  \dt \| \nabla_N {\cal G}^{(0)}_N e^{n+1} \|_2^2 .  \label{convergence-H1-12}
\end{eqnarray}
In turn, an application of discrete Gronwall inequality results in the desired convergence estimate:
\begin{eqnarray}
   \| \nabla_N {\cal G}^{(0)}_N e^{n+1} \|_2  + \Bigl( \frac18 \varepsilon^2 \dt   \sum_{m=1}^{n+1} \| \nabla_N \Delta_N e^m \|_2^2 \Bigr)^{1/2} \le C^{**} ( \dt^3 + h^m) .
   \label{convergence-H1-13}
\end{eqnarray}
In addition, by the preliminary estimate (\ref{prop-1-0-3}) (in Proposition~\ref{prop:prop 1}), we obtain the $\ell^\infty (0,T; H^1) \cap \ell^2 (0,T; H^3)$ error estimate:
\begin{eqnarray}
   \nrm{ \nabla_N e^{n+1} }_2  + \Bigl( \varepsilon^2 \dt   \sum_{m=1}^{n+1} \| \nabla_N \Delta_N e^m \|_2^2 \Bigr)^{1/2} \le \hat{C} ( \dt^3 + h^m) ,  \quad
   \mbox{with $\hat{C} = 2 \sqrt{2} C^{**}$} .
   \label{convergence-H1-14}
\end{eqnarray}

  Finally, we have to recover the a-priori assumption~\eqref{a priori-0} at time instant $t^{n+1}$, so that the analysis could be carried out in the induction style. The convergence estimate~\eqref{convergence-H1-14} indicates that
\begin{eqnarray}
    \| \nabla_N \Delta_N e^{n+1} \|_2  \le  \frac{\hat{C} \varepsilon^{-1} ( \dt^3 + h^m)}{\dt^{1/2}}   \le \hat{C} \varepsilon^{-1} ( \dt^{5/2} + \dt^{-1/2} h^m ) .  \label{a priori-2}
\end{eqnarray}
Since a singular $\dt^{-1}$ term appears on the right hand size, a scaling relation between $\dt$ and $h$ is needed in further analysis. If $\dt \ge h^m$ (which is a very relaxed condition), the following inequality is valid:
\begin{eqnarray}
    \| \nabla_N \Delta_N e^{n+1} \|_2  \le \hat{C} \varepsilon^{-1} ( \dt^{5/2} + h^{m/2} ) , \quad \mbox{if $\dt \ge h^m$}.  \label{a priori-3}
\end{eqnarray}
Otherwise, if $\dt \le h^m$, we apply an inverse inequality and obtain
\begin{eqnarray}
  \| \nabla_N \Delta_N e^{n+1} \|_2  &\le& \frac{ C \| \nabla_N e^{n+1} \|_2 }{h^2}
  \le \frac{ C \hat{C} (\dt^3 + h^m) }{h^2}  \le \frac{ C \hat{C} ( h^{3m} + h^m) }{h^2}  \nonumber
\\
  &\le&
  C \hat{C}  h^{m-2}  ,  \quad \mbox{if $\dt \le h^m$} .  \label{a priori-4}
\end{eqnarray}
Therefore, a combination of~\eqref{a priori-3} and \eqref{a priori-4} reveals that
\begin{eqnarray}
    \| \nabla_N \Delta_N e^{n+1} \|_2  \le \max \Bigl( \hat{C} \varepsilon^{-1} ( \dt^{5/2} + h^{m/2} ) ,  C \hat{C}  h^{m-2}  \Bigr) ,  \quad \mbox{for any $\dt$ and $h$} .  \label{a priori-5}
\end{eqnarray}
Consequently, the a-priori assumption~\eqref{a priori-0} could be recovered at time instant $t^{n+1}$ under the following constraint:
\begin{eqnarray}
  \dt \le (2 \hat{C})^{-2/5} \varepsilon^{2/5} ,\quad h \le (2 \hat{C})^{-2/m} \varepsilon^{2/m} .  \label{a priori-6}
\end{eqnarray}
We notice that the constraints for $\dt$ and $h$ are independent, and no scaling law between $\dt$ and $h$ is needed to pass through the analysis. In turn, this convergence estimate is unconditional. This validates the convergence estimate~\eqref{convergence-0-2}, and the proof for Theorem~\ref{thm:convergence} has been finished.


\section{Numerical results} \label{sec:numerical results}

\subsection{Convergence test for the numerical scheme}

In this subsection we perform a numerical accuracy check for the third order accurate ETD-based scheme~\eqref{scheme-ETD-3rd-0}-\eqref{scheme-ETD-3rd-1}. The computational domain is set to be $\Omega = (0,1)^2$, and the exact profile for the phase variable is set to be
	\begin{equation}
U (x,y,t) = \sin( 2 \pi x) \cos(2 \pi y) \cos( t) .
	\label{AC-1}
	\end{equation}
To make $U$ satisfy the original PDE \eqref{equation-NSS}, we have to add an artificial, time-dependent forcing term. Then the proposed third order numerical  scheme~\eqref{scheme-ETD-3rd-0}-\eqref{scheme-ETD-3rd-1} can be implemented to solve for (\ref{equation-NSS}).   We compute solutions with grid sizes $N=64$ to $N=192$ in increments of 16, and we solve up to time $T=1$.  The errors are reported at this final time.  
The surface diffusion parameter is taken as $\varepsilon=0.5$, the stabilization parameter is taken as $\kappa=\frac18$, and we set the artificial regularization parameter as $A=1$. The time step $\dt$ is determined by the linear refinement path $\dt = 0.5 h$, where $h$ is the spatial grid size. Figure~\ref{fig1} shows the discrete $L^1$, $L^2$ and $L^{\infty}$ norms of the errors between the numerical and exact solutions.  A clear third order accuracy is observed in all the norms.

	\begin{figure}
	\begin{center}
\includegraphics[width=3.0in]{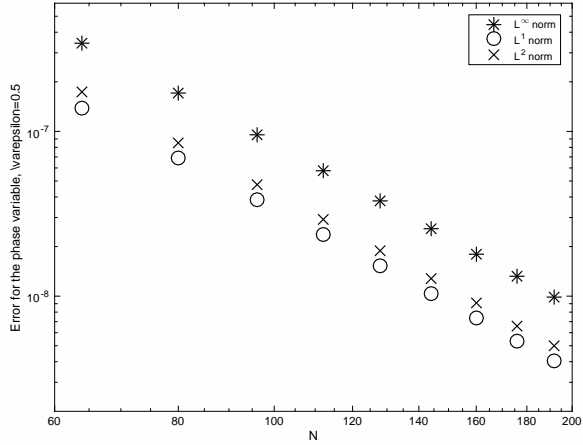}
	\end{center}
\caption{$L^1$, $L^2$ and $L^\infty$ numerical errors at $T=1.0$ plotted versus $N$ for the fully discrete second order scheme \eqref{scheme-ETD-3rd-0}-\eqref{scheme-ETD-3rd-1}.  The surface diffusion parameter is taken to be $\varepsilon=0.5$ and the time step size is $\dt = 0.5 h$.  The data lie roughly on curves $CN^{-3}$, for appropriate choices of $C$, confirming the full third-order accuracy of the scheme.}
	\label{fig1}
	\end{figure}


\subsection{Coarsening, energy dissipation and other physical quantities}

With the assumption that $\varepsilon\ll\min\left\{L_x,L_y\right\}$, how the solution to \eqref{equation-NSS} scales with time has always been of great interests. The physically interesting quantities that may be obtained from the solutions are (i) the energy $E(t)$; (ii) the characteristic (average) height (the surface roughness) $h(t)$;  and (iii) the characteristic (average) slope $m(t)$, the latter two defined precisely  as
	\begin{eqnarray}
h(t) &=&  \sqrt{ \frac{1}{| \Omega|} \int_{\Omega} \Bigl| u ( {\bf x}, t )  - \bar{u} (t) \Bigr|^2 \mathrm{d} {\bf x} }  \ ,  \quad  \mbox{with} \quad \,  \bar{u} (t) :=  \frac{1}{| \Omega|} \int_{\Omega}  u ( {\bf x}, t )  \mathrm{d} {\bf x} ,
 	\label{standard deviation}
	\\
m(t) &=& \sqrt{ \frac{1}{| \Omega|} \int_{\Omega}  \left| \nabla u ( {\bf x}, t ) \right|^2 \mathrm{d} {\bf x} } .
	\label{mound width}
	\end{eqnarray}
For the no-slope-selection equation~\eqref{equation-NSS}, one obtains $h\sim  O\left(t^{1/2}\right)$, $m(t) \sim O\left(t^{1/4}\right)$, and $E\sim O\left(-\ln(t)\right)$ as $t\to\infty$.  (See~\cite{golubovic97, libo03, libo04} and other related references.)  This implies that the characteristic (average) length $\ell(t) := h(t)/m(t)  \sim O\left( t^{1/4}\right)$ as $t\to\infty$. In other words, the average length and average slope scale the same with increasing time.  We also observe that the average mound height $h(t)$ grows faster than the average length $\ell(t)$, which is expected because there is no preferred slope of the height function $u$.

At a theoretical level, as described in~\cite{kohn06, kohn03, libo04}, one can at best  obtain lower bounds for the energy dissipation and, conversely, upper bounds for the average height.  However, the rates quoted as the upper or lower bounds are typically observed for the averaged values of the quantities of interest.  It is quite challenging to numerically predict these scaling laws, since very long time scale simulations are needed. To capture the full range of coarsening behaviors, numerical simulations for the coarsening process require short-time and long-time accuracy and stability, in addition to high spatial accuracy for small values of $\varepsilon$.	

In this article we display the numerical simulation results obtained from the proposed third order scheme~\eqref{scheme-ETD-3rd-0}-\eqref{scheme-ETD-3rd-1} for the no-slope-selection equation~\eqref{equation-NSS}, and compare the computed solutions against the predicted coarsening rates.  Similar results have also been reported for many first and second order accurate numerical schemes in the existing literature, such as the ones given by~\cite{chen12, Ju17, wang10}, etc. The surface diffusion coefficient parameter is taken to be $\varepsilon=0.02$ in this article, and the domain is set as $L=L_x = L_y = 12.8$. The uniform spatial resolution is given by $h = L /N$, $N=512$, which is adequate to resolve the small structures in the solution with such a value of $\varepsilon$. 

For the temporal step size $\dt$, we use increasing values of $\dt$, namely, $\dt = 0.004$ on the time interval $[0,400]$, $\dt = 0.04$ on the time interval $[400,6000]$, $\dt=0.16$ on the time interval $[6000, 3 \times 10^5]$. Whenever a new time step size is applied, we initiate the two-step numerical scheme by  taking $u^{-1} = u^{-2} = u^0$, with the initial data $u^0$ given by the final time output of the last time period. Figure~\ref{fig3} displays time snapshots of the film height $u$ with $\varepsilon=0.02$, with significant coarsening observed in the system.  At early times many small hills (red) and valleys (blue) are present.  At the final time, $t= 300000$, a one-hill-one-valley structure emerges, and further coarsening is not possible.

\begin{figure}[h]
	\begin{center}
\includegraphics[width=6.0in]{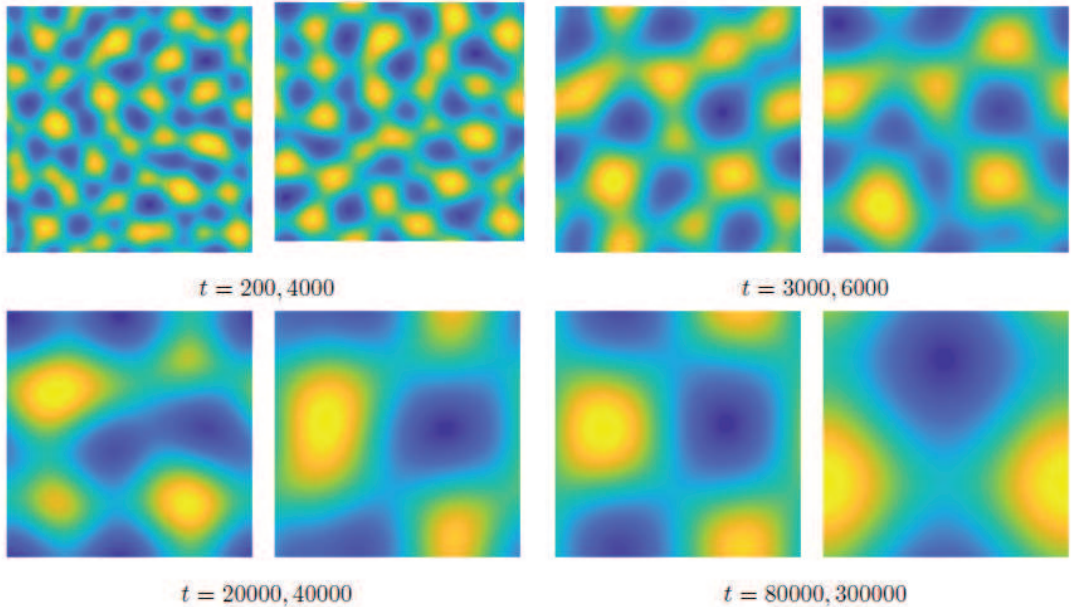}
\caption{(Color online.) Snapshots of the computed height function $u$ at the indicated times for the parameters $L =12.8$, $\varepsilon = 0.02$.  
The hills at early times are not as high as time at later times, and similarly with the valley. The average height/depth evolution with time could be seen in Figure~\ref{fig4}.}		
		\label{fig3}
	\end{center}
\end{figure}


The long time characteristics of the solution, especially the energy decay rate, average height growth rate, and the mound width growth rate, are of interest to surface physics community.  The last two quantities can be easily measured experimentally. On the other hand, the discrete energy $E_N$ is defined via~\eqref{energy-discrete-spectral}; the space-continuous average height and average slope have been defined in \eqref{standard deviation}, \eqref{mound width}, and the analogous discrete versions are also available. 
Theoretically speaking, the lower bound for the energy decay rate is of the order of $- \ln(t)$, and the upper bounds for the average height  and average slope/average length are of the order of $t^{1/2}$, $t^{1/4}$, respectively, as established for the no-slope-selection equation~\eqref{equation-NSS} in~\cite{libo04}.  Figure~\ref{fig4} presents the semi-log plots for the energy versus time and log-log plots for the average height versus time, and average slope versus time, respectively, with the given physical parameter $\varepsilon=0.02$. The detailed scaling ``exponents" are obtained using least squares fits of the computed data up to time $t=400$.  A clear observation of the $- \ln(t)$, $t^{1/2}$ and $t^{1/4}$ scaling laws can be made, with different coefficients dependent upon $\varepsilon$, or, equivalently, the domain size, $L$.

	\begin{figure}
	\centering
        \hbox{
	\includegraphics[width=2.0in,height=2.0in]{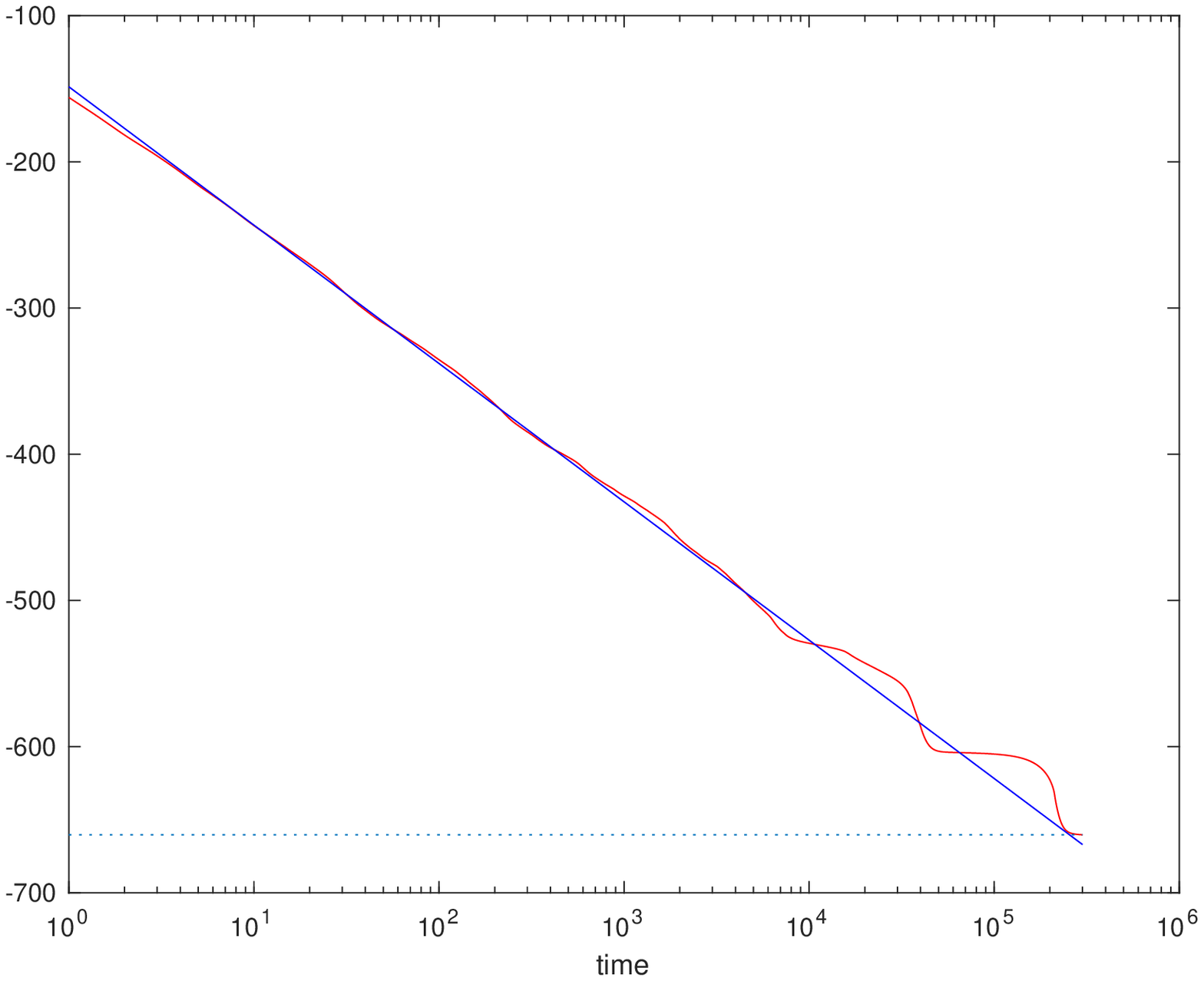}  \hskip 0.2cm
\includegraphics[height=2.0in,width=2.0in]{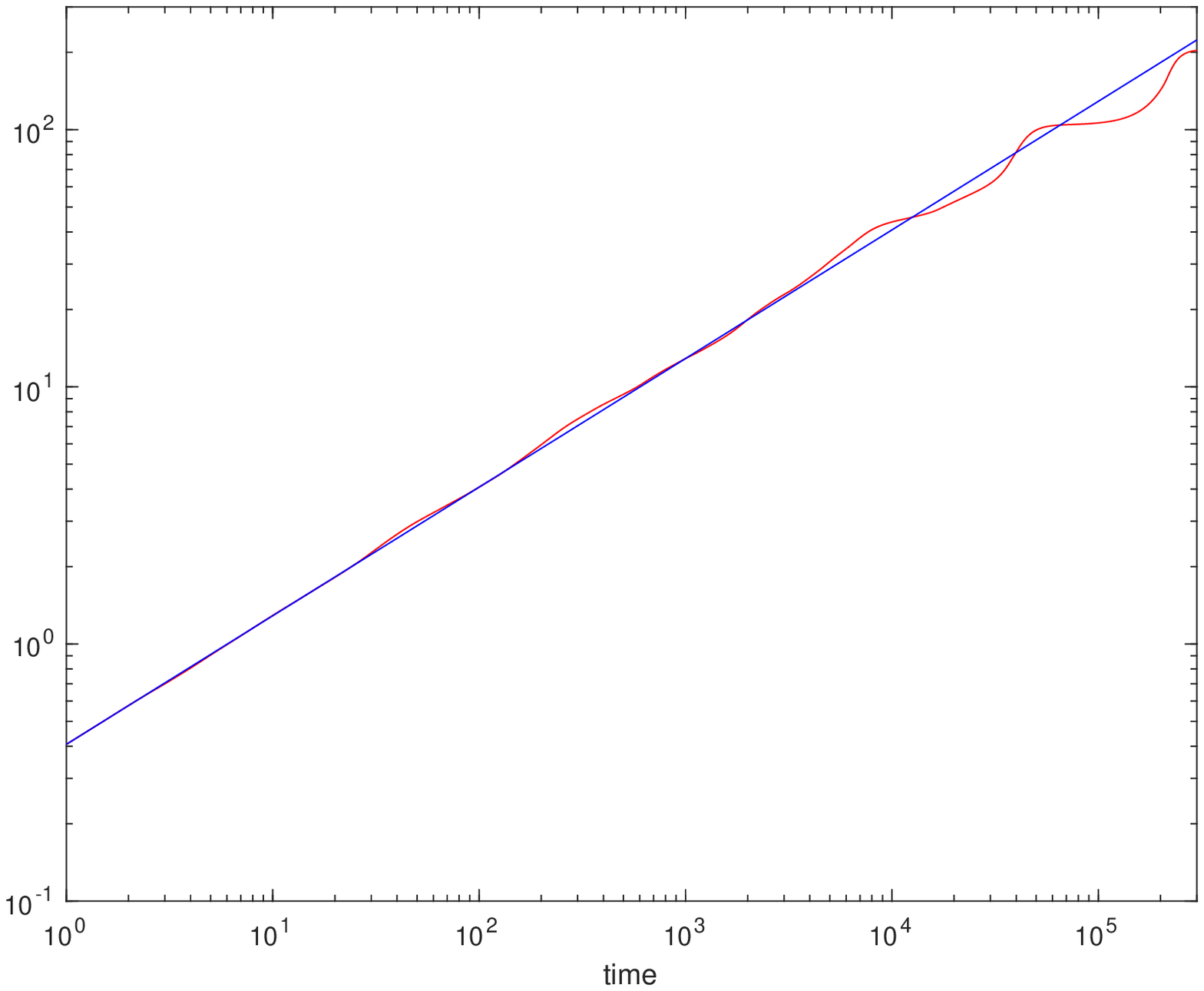} \hskip 0.2cm
\includegraphics[height=2.0in,width=2.0in]{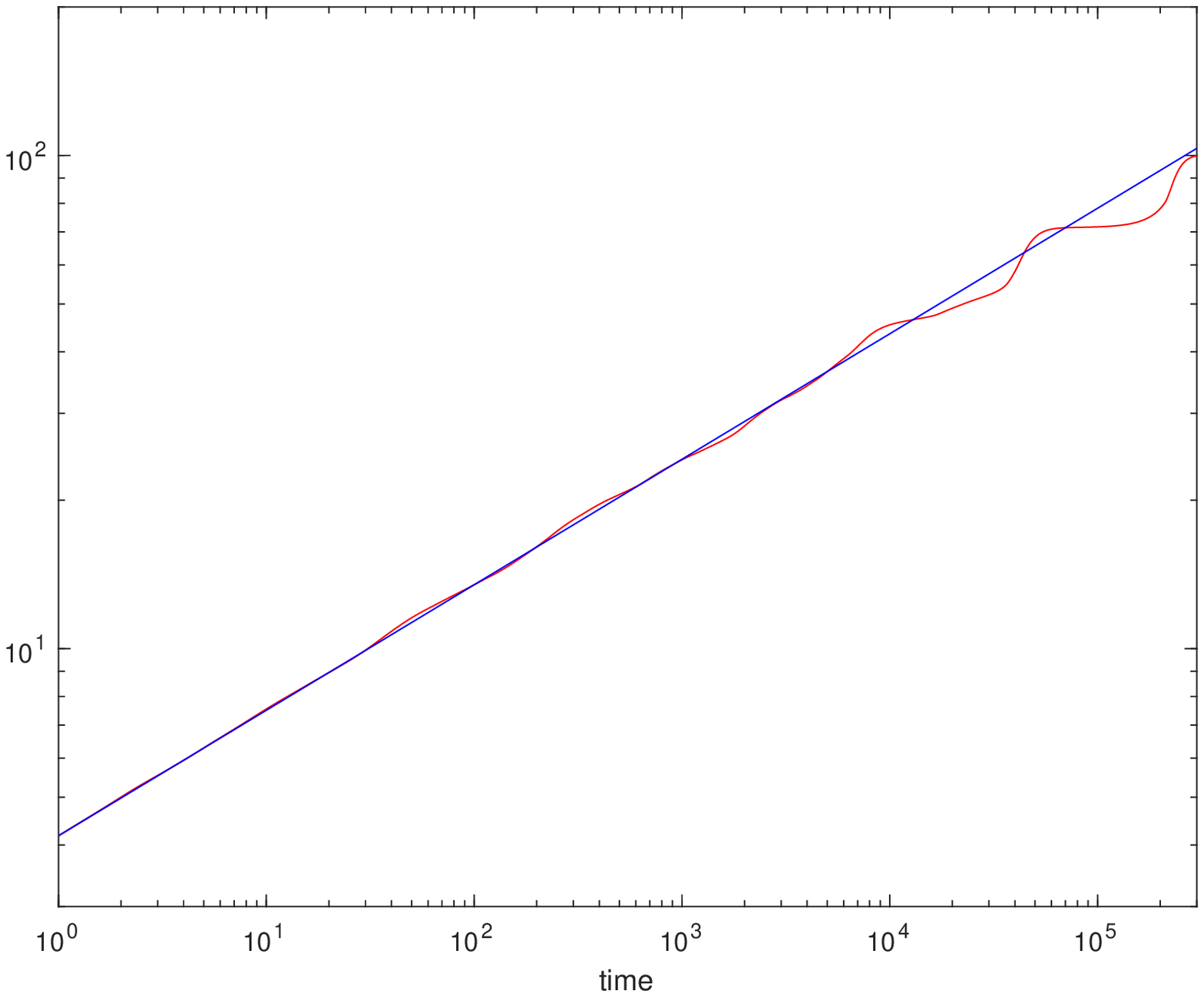} }
\caption{Left: Semi-log plot of the temporal evolution the energy $E_N$ for $\varepsilon=0.02$.  The energy decreases like $-\ln(t)$ until saturation. Middle: The log-log plot of the average height (or roughness) of $u$, denoted as $h(t)$, which grows like $t^{1/2}$. Right: The log-log plot of the average width of $u$, denoted $m(t)$, which grows like $t^{1/4}$. The dotted lines correspond to the minimum energy reached by the numerical simulation. The red lines represent the energy plot obtained by the simulations, while the straight lines are obtained by least squares approximations to the energy data.  The least squares fit is only taken for the linear part of the calculated data, only up to about time $t=400$.  The fitted line for the energy has the form $a_e\ln(t)+b_e$, with $a_e = -41.0983$, $b_e=-148.6410$; the (blue) fitting line for the average height has the form $a_ht^{b_h}$, with  $a_h =   0.4071$, $b_h = 0.5001$, and the fitting line for the average width has the form $a_m t^{b_m}$, with $a_m = 4.1747$, $b_m = 0.2545$.}
	\label{fig4}
       \end{figure}

Now we recall that a lower bound for the energy (\ref{energy-NSS}), assuming $\Omega = (0,L)^2$, which has been derived and polished in our earlier works~\cite{chen12, chen14, wang10}:
	\begin{equation}
E (\phi) \ge \frac{L^2}{2}\left( \ln\left(\frac{4 \varepsilon^2\pi^2}{L^2}\right)-\frac{4\varepsilon^2\pi^2}{L^2}+1\right) =:\gamma  \ .
	\label{lower-bound}
	\end{equation}
Obviously, since the energy is bounded below it cannot keep decreasing at the rate $-\ln(t)$.  This fact manifests itself in the calculated data as the rate of decrease of the energy, for example, begins to wildly deviate from the predicted $-\ln(t)$ curve.  Sometimes the rate of decrease increases, and sometimes it slows as the systems ``feels" the periodic boundary conditions.  Interestedly, regardless of this later-time deviation from the accepted rates, the time at which the system saturates (\emph{i.e.}, the time when the energy abruptly and essentially stops decreasing) is roughly that predicted by extending the blue lines in Figure~\ref{fig4} to the predicted minimum energy~\eqref{lower-bound}.

	\begin{rem}
In this presented numerical simulation, the spatial resolution and time step sizes are taken as the same as the ones presented for the second order energy stable scheme~\cite{chen14}. Meanwhile, since a linear iteration algorithm has to be applied for the highly nonlinear numerical scheme in~\cite{chen14}, the computational cost at each time step is about 3 to 5 times as the proposed 3rd order ETD-based one. 
For the long time simulation, both numerical schemes have produced similar evolutionary curves in terms of energy, standard deviation, and the mound width. A more detailed calculation shows that long time asymptotic growth rate of the standard deviation given by the third order numerical simulation is closer to $t^{1/2}$ than that by the second order energy stable scheme: $m_r = 0.5001$, as recorded in Figure~\ref{fig4}, while in~\cite{chen14} this exponent was found to be $m_r = 0.5132$. Similarly, the long time asymptotic growth rate of the mound width given by~\eqref{scheme-ETD-3rd-0}-\eqref{scheme-ETD-3rd-1} is closer to $t^{1/4}$ than that by the second order energy stable scheme in: $b_m = 0.2545$, as recorded in Figure~\ref{fig4}, in comparison with $m_r = 0.2607$ reported in~\cite{chen14}. This gives more evidence that the third order scheme is able to produce more accurate long time numerical simulation results than the second order schemes, even if the computational cost is even less than the one given by~\cite{chen14}, due to the linear iteration algorithm to implement the nonlinear numerical scheme.

Similar comparison has also been made between the second order ETD-related scheme and the proposed third order ETD-based scheme, given by~\eqref{scheme-ETD-2nd} (outlined in~\cite{Ju17}) and ~\eqref{scheme-ETD-3rd-0}-\eqref{scheme-ETD-3rd-1}, respectively: $m_r = 0.5001$, $b_m=0.2545$ for the proposed third order scheme, in comparison with $m_r=0.510$, $b_m=0.258$, for $ETDMs2$, as reported in~\cite{Ju17}. This gives another evidence of robustness of the third order accurate numerical scheme for the NSS equation~\eqref{equation-NSS}.
	\end{rem}

\section{Concluding remarks} \label{sec:conclusion}

In this article, we propose and analyze a third order accurate ETD-based numerical scheme for the NSS equation~\eqref{equation-NSS} of the epitaxial thin film growth model, combined with Fourier pseudo-spectral spatial discretization. An exact integration of the linear part of the NSS equation is involved in the ETD-based scheme, followed by multi-step explicit approximation of the temporal integral of the nonlinear term. More importantly, a third order accurate Douglas-Dupont regularization term is added in the numerical scheme. In turn, a careful Fourier eigenvalue analysis leads to the energy stability in a modified version, and a uniform in time bound of the numerical energy becomes available. Furthermore, the optimal rate convergence analysis and error estimate are derived in details, with a decomposition of the numerical scheme into two stages. Error estimates are carried out in both stages, with extensive applications of linearized stability analysis in the second stage. To overcome the difficulties associated with many global operators involved in the algorithm, as well as their inverse operators, we have to perform careful eigenvalue estimates for these operators, as well as their composition. In addition, 
an aliasing error control technique has to be utilized in the $\ell^\infty (0, T, H_h^1) \cap \ell^2 (0,T; H_h^3)$ error estimate, combined with extensive scaling law arguments between $\dt$ and $h$. This convergence estimate is the first such result for a third order accurate scheme for a gradient flow. Some numerical simulation results are presented to demonstrate the robustness of the numerical scheme and the third order convergence. In particular, the long time simulation results have revealed that, the power index for the surface roughness and the mound width growth for $\varepsilon=0.02$ (up to $T=3 \times 10^5$), created by the proposed third order ETD-based numerical scheme, is more accurate than these created by certain second order accurate, energy stable schemes in the existing literature.

	\section*{Acknowledgements}
This work is supported in part by the Longshan Talent Project of SWUST 18LZX529  (K. Cheng), Hong Kong Research Council GRF grants 15300417 and 15325816,
(Z. Qiao) and NSF DMS-1418689 (C.~Wang). 

	\appendix

	\section{Proof of Lemma \ref{lem:lem 1}}
	\label{proof:Lemma 1}
	
First, we review the following estimate in Calculus.

\begin{lem}  \label{lem:prelim lem}
Suppose that $f(x)$ and $g(x)$ are continuous functions, $f (x) >0$, $g (x) >0$, and $\frac{f (x)}{g (x)}$ is decreasing over $(0, + \infty)$. Define $H(x) = \frac{\int_0^x f (t) dt}{\int_0^x g(t) dt}$. Then $H(x)$ is decreasing over $(0, + \infty)$.
\end{lem}

\begin{proof}
  Denote $F(x) = \int_0^x f (t) dt$, $G(x) = \int_0^x g (t) dt$, so that $H(x) = \frac{F(x)}{G(x)}$. For any $0 < x_1 < x_2$, we make a comparison between $H(x_1)$ and $H (x_2)$.

  We denote $C_1 = \frac{f(x_1)}{g(x_1)}$. By the decreasing property of $\frac{f (x)}{g (x)}$, we see that $\frac{f (x)}{g (x)} \ge C_1$ for $0\le x \le x_1$, and $\frac{f (x)}{g (x)} \le C_1$ for $x_1 \le x  \le x_2$. This in turn implies that
\begin{equation}
  f (x) \ge C_1 g(x) ,  \, \, \, \mbox{for $0\le x \le x_1$} , \quad
  f (x) \le C_1 g(x) ,  \, \, \, \mbox{for $x_1\le x \le x_2$} .
\end{equation}
As a result, we get
\begin{equation}
  \int_0^{x_1} f (t) dt \ge C_1 \int_0^{x_1} g (t) dt , \quad
  \int_{x_1}^{x_2} f (t) dt \le C_1 \int_{x_1}^{x_2} g (t) dt .
\end{equation}
In turn, if we denote $A_1 =  \int_0^{x_1} f (t) dt$, $B_1 =  \int_0^{x_1} g (t) dt$, $A_2 =  \int_{x_1}^{x_2} f (t) dt$, $B_2 =  \int_{x_1}^{x_2} g (t) dt$, we have
\begin{equation}
  \frac{A_1}{B_1} \ge C_1  \ge \frac{A_2}{B_2} .
\end{equation}
Then we arrive at
\begin{equation}
  H(x_2) = \frac{\int_0^{x_2} f (t) dt}{\int_0^{x_2} g(t) dt} = \frac{A_1 + A_2}{B_1 + B_2} \le \frac{A_1}{B_1} = H(x_1).
\end{equation}
This completes the proof for the decreasing property of $H (x)$.
\end{proof}

Next, we proceed into the proof of Lemma \ref{lem:lem 1}.

\begin{proof}
The function $g_0 (x)$ could be represented as
\begin{eqnarray}
  g_0 (x) = \frac{1 - {\rm e}^{-x}}{x} = \frac{\int_0^x {\rm e}^{-t} \, dt}{\int_0^x 1 \, dt} .
  \label{lem 1-1}
\end{eqnarray}
On the other hand, $\frac{{\rm e}^{-x}}{1} = {\rm e}^{-x}$ is a decreasing function over $(0, +\infty)$. By Lemma~\ref{lem:prelim lem}, we conclude that $g_0 (x)$ is decreasing.

Similarly, $g_1 (x)$ could be rewritten as
\begin{eqnarray}
  g_1 (x) = \frac{x - (1 - {\rm e}^{-x} )}{x^2} = \frac{\int_0^x ( 1 - {\rm e}^{-t} ) \, dt}{\int_0^x 2t \, dt} . \label{lem 1-2}
\end{eqnarray}
Since $\frac{1 - {\rm e}^{-x}}{2 x} = \frac12 g_0 (x)$ is a decreasing function over $(0, +\infty)$, an application of Lemma~\ref{lem:prelim lem} reveals that $g_1 (x)$ is also decreasing.

For $g_2 (x)$, we look at its rewritten form
\begin{eqnarray}
  g_2 (x) = \frac{x^2 - 2 (x - (1 - {\rm e}^{-x} ) )}{x^3}  = \frac{\int_0^x 2 ( t - (1 - {\rm e}^{-t} ) ) \, dt}{\int_0^x 3 t^2 \, dt} . \label{lem 1-3}
\end{eqnarray}
Since $\frac{2 ( x - (1 - {\rm e}^{-x} ) )}{3 x^2} = \frac23 g_1 (x)$ is a decreasing function over $(0, +\infty)$, an application of Lemma~\ref{lem:prelim lem} reveals that $g_2 (x)$ is also decreasing. This finishes the proof of the first part of Lemma~\ref{lem:lem 1}.

As a direct consequence of their decreasing property, we see that $g_0 (x) \le g_0 (0) = 1$, $g_1 (x) \le g_1 (0) = \frac12$ and $g_2 (x) \le g_2 (0) = \frac13$, \, \, $\forall x > 0$.

In turn, we observe that
\begin{eqnarray}
  &&
  \frac{g_1 (x)}{g_0 (x)} \le \frac{g_1 (0)}{g_0 (2)} = \frac{\frac12}{\frac{1 - {\rm e}^{-2}}{2} }  = \frac{1}{1 - {\rm e}^{-2}} ,  \quad \mbox{for $x \le 2$} ,
\\
  &&
  \frac{g_1 (x)}{g_0 (x)} = \frac{1 - \frac{1 - {\rm e}^{-x}}{x} }{1 - {\rm e}^{-x} }
  \le \frac{1}{1 - {\rm e}^{-2} } ,  \quad \mbox{for $x \ge 2$}  .
\end{eqnarray}
For the function $\frac{g_2 (x)}{g_0 (x)}$, we have
\begin{eqnarray}
  &&
  \frac{g_2 (x)}{g_0 (x)} \le \frac{g_2 (0)}{g_0 (2)} = \frac{\frac13}{\frac{1 - {\rm e}^{-2}}{2} }  = \frac{2}{3(1 - {\rm e}^{-2})} ,  \quad \mbox{for $x \le 2$} ,
\\
  &&
  \frac{g_2 (x)}{g_0 (x)} = \frac{1 - 2 \frac{1 - \frac{1 - {\rm e}^{-x}}{x} }{x}  }{1 - {\rm e}^{-x} }  \le \frac{1}{1 - {\rm e}^{-2} } ,  \quad \mbox{for $x \ge 2$}  .
\end{eqnarray}
Therefore, the inequalities $\frac{g_1 (x)}{g_0 (x)} \le \frac{1}{1 - {\rm e}^{-2} }$, $\frac{g_2 (x)}{g_0 (x)} \le \frac{1}{1 - {\rm e}^{-2} }$ are valid. This finishes the proof of Lemma~\ref{lem:lem 1}.
\end{proof}

\section{Proof of Proposition~\ref{prop:prop 1} } \label{proof:Prop 1}

An application of Parseval equality to the discrete Fourier expansions for $f$ and ${\cal G}^{(0)}_N f$, given by~\eqref{Fourier-1} and \eqref{Fourier-3}, respectively, leads to
\begin{eqnarray}
  \| f \|_2^2 = L^2 \sum_{k,\ell=-K}^K  | \hat{f}_{k,\ell} |^2 ,  \quad
  \nrm{ {\cal G}^{(0)}_N f }_2^2 = L^2 \sum_{k,\ell=-K}^K  \frac{\dt \Lambda_{k,\ell}}{1 - {\rm e}^{- \dt \Lambda_{k,\ell}} } | \hat{f}_{k,\ell} |^2 .  \label{prop-1-1}
\end{eqnarray}
Meanwhile, the following observation is made:
\begin{eqnarray}
  \frac{\dt \Lambda_{k,\ell}}{1 - {\rm e}^{- \dt \Lambda_{k,\ell}} }  = \frac{1}{g_0 (\dt \Lambda_{k,\ell})} . \label{prop-1-2-1}
\end{eqnarray}
With an application of Lemma~\ref{lem:lem 1}, we obtain
\begin{eqnarray}
   1 = \frac{1}{g_0 (0)} \le \frac{1}{g_0 (x) } = \frac{x}{ 1 - {\rm e}^{-x} }  \le 1+x ,  \quad \forall x > 0 . \label{prop-1-2-2}
\end{eqnarray}
This in turn implies that
\begin{eqnarray}
  1 \le \frac{\dt \Lambda_{k,\ell}}{1 - {\rm e}^{- \dt \Lambda_{k,\ell}} }  \le 1+ \dt \Lambda_{k,\ell} ,  \quad \mbox{for any $k$, $\ell$} . \label{prop-1-3}
\end{eqnarray}
Its combination with~\eqref{prop-1-1} reveals that
\begin{eqnarray}
  \| f \|_2^2 \le \nrm{ {\cal G}^{(0)}_N f }_2^2 \le L^2 \sum_{k,\ell=-N}^N ( 1+ \dt \Lambda_{k,\ell})  | \hat{f}_{k,\ell} |^2 = \| f \|_2^2 + \dt ( \varepsilon^2 \| \Delta_N f \|_2^2 + \kappa \| \nabla_N f \|_2^2 ) ,  \label{prop-1-4}
\end{eqnarray}
which in turn results in~\eqref{prop-1-0-1}, by taking $C_1 = 1$.

  The proof of the first inequality of~\eqref{prop-1-0-1-2} follows a similar form of Fourier analysis, combined with the following identity:
\begin{eqnarray}
  \frac{\dt \Lambda_{k,\ell}}{1 - {\rm e}^{- \dt \Lambda_{k,\ell}} }  \ge \dt \Lambda_{k,\ell} ,  \quad \forall k, \ell .
\end{eqnarray}
For the second inequality of~\eqref{prop-1-0-1-2}, we begin with the following identities
\begin{eqnarray}
  &&
  (L_N  f)_{i,j} = \sum_{k,\ell=-K}^K \Lambda_{k, \ell} \hat{f}_{k,\ell} {\rm e}^{2 \pi i ( k x_i + \ell y_j)/L}  ,   \nonumber
\\
  &&
  (- \Delta_N {\rm e}^{- L_N \dt} f )_{i,j} = \sum_{k,\ell=-K}^K ( - \lambda_{k, \ell}) {\rm e}^{- \dt \Lambda_{k, \ell} } \hat{f}_{k,\ell} {\rm e}^{2 \pi i ( k x_i + \ell y_j)/L} ,  \nonumber
\end{eqnarray}
so that
\begin{eqnarray}
   \left\langle L_N  f , - \Delta_N {\rm e}^{- L_N \dt} f  \right\rangle
   = L^2 \sum_{k, \ell=-K}^K ( - \lambda_{k, \ell}) \Lambda_{k, \ell}
   {\rm e}^{- \dt \Lambda_{k, \ell} }  | \hat{f}_{k,\ell} |^2 \ge 0 ,
\end{eqnarray}
since $- \lambda_{k, \ell} \ge 0$, $\Lambda_{k, \ell} \ge 0$, and ${\rm e}^{- \dt \Lambda_{k, \ell} } \ge 0$, for any $k, \ell$.

  The proof of \eqref{prop-1-0-2} and \eqref{prop-1-0-3} could be carried out in the same manner; the details are left to interested readers.

   For the analysis of $G^{(1)}_N f$, with its discrete Fourier expansion given by~\eqref{Fourier-2-2}, an application of Parseval equality gives
\begin{eqnarray}
  \| f \|_2^2 = L^2 \sum_{k,\ell=-K}^K  | \hat{f}_{k,\ell} |^2 ,  \quad
  \nrm{ G^{(1)}_N f  }_2^2 = L^2 \sum_{k,\ell=-K}^K  \Bigl( \frac{1 - \frac{1 - {\rm e}^{- \dt \Lambda_{k,\ell}} }{\dt \Lambda_{k,\ell}} }{1 - {\rm e}^{- \dt \Lambda_{k,\ell}} }   \Bigr)^2
   | \hat{f}_{k,\ell} |^2 .  \label{prop-1-5}
\end{eqnarray}
On the other hand, the following observation is available:
\begin{eqnarray}
     \frac{1 - \frac{1 - {\rm e}^{- \dt \Lambda_{k,\ell}} }{\dt \Lambda_{k,\ell}} }{1 - {\rm e}^{- \dt \Lambda_{k,l}} } = \frac{ g_1 (\dt \Lambda_{k,\ell}) }{ g_0 (\dt \Lambda_{k,\ell}) }
     \le \frac{1}{1 - {\rm e}^{-2}} ,
        \label{prop-1-6}
\end{eqnarray}
with an application of Lemma~\ref{lem:lem 1} in the second step. Its substitution into~\eqref{prop-1-5} results in
\begin{eqnarray}
  \| f \|_2^2 = L^2 \sum_{k,\ell=-K}^K  | \hat{f}_{k,\ell} |^2 ,  \quad
  \nrm{ G^{(1)}_N f  }_2^2 \le ( \frac{1}{1 - {\rm e}^{-2}} )^2 \| f \|_2^2 .  \label{prop-1-7}
\end{eqnarray}
Then we have proved the first inequality of~\eqref{prop-1-0-4}, with $C_4 = \frac{1}{1 - {\rm e}^{-2}}$.

  The analysis of $G^{(2)}_N f$ could be carried out in the same manner, and the second inequality of~\eqref{prop-1-0-4} is with $C_5 = \frac{1}{1 - {\rm e}^{-2}}$. This finishes the proof of Proposition~\ref{prop:prop 1}.

\section{Proof of Proposition~\ref{prop:prop 2} } \label{proof:Prop 2}


We begin with the following expansion:
\begin{eqnarray}
  \nabla_N \left( f_N (U^k) - f_N (u^k) \right) = \nabla_N \nabla_N \cdot \left( g_N (U^k) - g_N (u^k) \right) . \label{prop 2-4}
\end{eqnarray}
Meanwhile, we denote $U_N^k$, $u_N^k$ and $e_N^k$ as the continuous extension of $U^k$, $u^k$ and $e^k$, as the formula given by~\eqref{spectral-coll-projection-2}. We notice that $g_N$ is defined in the sense of collocation way, at a point-wise level. Since $g_N (U^k) - g_N (u^k)$ is the grid point interpolation of $g (U_N^k) - g (u_N^k)$, 
we apply~\eqref{spectral-coll-projection-5} (in Lemma~\ref{lemma:aliasing error-2}) to control the aliasing error. Subsequently, we arrive at
\begin{eqnarray}
  &&
  \| \nabla_N \nabla_N \cdot \left( g_N (U^k) - g_N (u^k) \right) \|_2
  = \| \nabla \nabla \cdot P_c^N \left( g (U_N^k) - g (u_N^k) \right) \|  \nonumber
\\
  &\le&
  C  \| P_c^N \left( g (U_N^k) - g (u_N^k) \right) \|_{H^2}
  \le C  \| g (U_N^k) - g (u_N^k) \|_{H^2} , \label{prop 2-5}
\end{eqnarray}
in which the fact that $2 > \frac{d}{2} =1$ has been used. On the other hand, we have the following expansion for $g (U_N^k) - g (u_N^k)$, in a similar form as~\eqref{scheme-ETD-3rd-stability-4-2}:
\begin{eqnarray}
   g (U_N^k) - g (u_N^k) =   \frac{\nabla e_N^k}{1+|\nabla u_N^k|^2}  +  \frac{ \nabla (U_N^k + u_N^k) \cdot \nabla e_N^k}{(1+|\nabla U_N^k |^2) (1+|\nabla u_N^k|^2)} \nabla U_N^k + \kappa \nabla e_N^k . \label{prop 2-6}
\end{eqnarray}
A repeated application of H\"older inequality and Sobobev inequality leads to the following estimates:
\begin{eqnarray}
   &&
   \| \frac{\nabla e_N^k}{1+|\nabla u_N^k|^2}  + \kappa \nabla e_N^k \|_{H^2}
   \le C ( \| U_N^k \|_{H^3}^2 + \| u_N^k \|_{H^3}^2 )   ( \| \nabla e_N^k \| + \| \Delta e_N^k \| + \| \nabla \Delta e_N^k \| ) ,  \label{prop 2-7-1}
\\
  &&
   \| \frac{ \nabla (U_N^k + u_N^k) \cdot \nabla e_N^k}{(1+|\nabla U_N^k |^2) (1+|\nabla u_N^k|^2)} \nabla U_N^k  \|_{H^2}  \nonumber
\\
  &&  \qquad
   \le C ( \| U_N^k \|_{H^3}^2 + \| u_N^k \|_{H^3}^2 )   ( \| \nabla e_N^k \| + \| \Delta e_N^k \| + \| \nabla \Delta e_N^k \| ) .  \label{prop 2-7-2}
\end{eqnarray}
Meanwhile, with the a-priori assumption~\eqref{a priori-0}, we have
\begin{eqnarray}
    \| U_N^k \|_{H^3} \le C*, \quad \| u_N^k \|_{H^3} \le \| U_N^k \|_{H^3}  + \| e_N^k \|_{H^3} \le \| U_N^k \|_{H^3}  + C_6 \| \nabla \Delta e_N^k \| \le C^* + C_6 := \tilde{C}_1 ,
\end{eqnarray}
in which $C_6$ is a constant associated with elliptic regularity: $\| e_N^k \|_{H^3} \le C_6 \| \nabla \Delta e_N^k \|$, since $\int_\Omega e_N^k \, d {\bf x} =0$. Then we arrive at
\begin{eqnarray}
  &&
  \| \nabla_N \left( f_N (U^k) - f_N (u^k) \right) \|_2 = \| \nabla_N \nabla_N \cdot \left( g_N (U^k) - g_N (u^k) \right) \|_2 \le  C  \| g (U_N^k) - g (u_N^k) \|_{H^2}  \nonumber
\\
  &\le&
   C ( \| U_N^k \|_{H^3}^2 + \| u_N^k \|_{H^3}^2 )   ( \| \nabla e_N^k \| + \| \Delta e_N^k \| + \| \nabla \Delta e_N^k \| )  \nonumber
\\
  &\le&
     C ( (C^*)^2 + \tilde{C}_1^2 ) C_6 \| \nabla \Delta e_N^k \|
   \le C ( (C^*)^2 + \tilde{C}_1^2 ) C_6 \| \nabla_N \Delta_N e^k \| . \label{prop 2-8}
\end{eqnarray}
As a result, \eqref{prop 2-1} has been established, by taking $C_0^{(2)} = C ( (C^*)^2 + \tilde{C}_1^2 ) C_6$. This finishes the proof of Proposition~\ref{prop:prop 2}.


	\bibliographystyle{plain}

	\end{document}